\documentclass[11pt,a4paper]{amsart}
\usepackage[T1]{fontenc}
\usepackage[utf8]{inputenc}
\usepackage{lmodern,amsmath,amssymb,amsthm,mathtools,mathrsfs,txfonts}
\usepackage[margin=35mm]{geometry}
\usepackage{microtype,booktabs,array,enumitem,xcolor}
\usepackage{tikz-cd}
\tikzcdset{
  tikzcd to/.tip={cm to[scale=0.75]},
  tikzcd double to/.tip={cm to[scale=0.75,sep=+0pt -1.95]cm to[scale=0.75]},
  tikzcd left hook/.tip={cm left hook[scale=0.75]}}
\usepackage[colorlinks=true,linkcolor=black,citecolor=blue!50!black,urlcolor=blue!50!black]{hyperref}
\usepackage[nameinlink,noabbrev]{cleveref}

\setlist{itemsep=3pt,topsep=5pt}
\newcolumntype{P}[1]{>{\raggedright\arraybackslash}p{#1}}
\newtheorem{theorem}{Theorem}[section]
\newtheorem{maintheorem}{Theorem}

\crefname{maintheorem}{Theorem}{Theorems}
\Crefname{maintheorem}{Theorem}{Theorems}
\newtheorem{proposition}[theorem]{Proposition}
\newtheorem{lemma}[theorem]{Lemma}
\newtheorem{corollary}[theorem]{Corollary}
\theoremstyle{definition}

\newtheorem{remark}[theorem]{Remark}

\crefname{theorem}{Theorem}{Theorems}
\crefname{proposition}{Proposition}{Propositions}
\crefname{lemma}{Lemma}{Lemmas}
\crefname{corollary}{Corollary}{Corollaries}
\crefname{definition}{Definition}{Definitions}
\crefname{example}{Example}{Examples}
\crefname{remark}{Remark}{Remarks}
\crefname{section}{Section}{Sections}

\newcommand{\R}{\mathbb R}
\newcommand{\Q}{\mathbb Q}
\newcommand{\Z}{\mathbb Z}
\newcommand{\C}{\mathbb C}

\newcommand{\HH}{\mathcal H}
\newcommand{\Mod}{\operatorname{Mod}}
\newcommand{\Sp}{\operatorname{Sp}}
\newcommand{\SL}{\operatorname{SL}}
\newcommand{\GL}{\operatorname{GL}}
\newcommand{\Aut}{\operatorname{Aut}}
\newcommand{\Hom}{\operatorname{Hom}}
\newcommand{\tor}{\operatorname{tor}}

\newcommand{\Ad}{\operatorname{Ad}}

\newcommand{\im}{\operatorname{im}}
\newcommand{\ab}{\mathrm{ab}}
\newcommand{\rt}{(T)}
\newcommand{\eps}{\varepsilon}

\newcommand{\Ga}{\mathbf G_a}

\newcommand{\un}{\mathrm{un}}
\title{Property \texorpdfstring{$\rt$}{(T)} and nonlinearity for mapping class group quotients}
\author{Piotr W. Nowak}

\begin{document}
\maketitle
\begin{abstract}
We give a general method for proving Kazhdan's property~\(\rt\)
for quotients \(G/K_{[c+1]}\), where \(G\) is countable, \(K\) is a normal subgroup
and \(K_{[j]}\) denotes its lower central series. The method
combines an affine realization of \(G/[K,K]\), a contraction
argument, and permanence for nilpotent normal subgroups.

We apply this method to several classes of groups by means of  homological representations.
For the Torelli lower-central series quotients
\(\Mod(\Sigma_g)/(\mathcal{T}_g)_{[c+1]}\) for \(g\ge3\) we show that they have property~\(\rt\) for every \(c\ge1\).
We also prove
property~\(\rt\) for
\(\Aut(F_3)/(\mathrm{IA}_3)_{[c+1]}\) for every \(c\ge1\), relating
these groups to the tame nilpotent images studied by
Lubotzky and Pak. 
Several other classes of examples are discussed as well.  

We also prove that the Torelli lower-central series quotients for \(c\ge2\) are not residually finite, and thus non-linear. 
\end{abstract}

\section{Introduction}

Kazhdan's property~\(\rt\) is a rigidity property of unitary
representations with many spectacular applications across mathematics. Classically, the examples coming from higher-rank Lie groups
and their lattices  \cite{Kazhdan1967} have been an important source of both results and
questions in group theory, combinatorics, noncommutative geometry and dynamics, among others. We refer to \cite{BHV} for an overview of property \(\rt\) and its applications.

Two natural families of interest in this context are automorphism groups of free
groups and mapping class groups. 
For the free group \(F_n\) of rank \(n\), the group \(\Aut(F_n)\)
has property~\(\rt\) for every \(n\ge4\): see \cite{KalubaNowakOzawa} for the case \(n=5\), \cite{KalubaKielakNowak} for the case \(n\ge 5\) and \cite{NitscheAutFour} for \(n=4\).
In rank three, \(\Aut(F_3)\) does not have property~\(\rt\),
as follows from work of McCool \cite{McCool} and Grunewald and Lubotzky
\cite{GrunewaldLubotzky}.
For a closed connected oriented surface \(\Sigma_g\) of genus
\(g\ge3\), it remains open whether \(\Mod(\Sigma_g)\)
has property~\(\rt\).

In this paper we prove property \(\rt\) for new examples of groups arising as quotients of mapping class groups and automorphism groups of free groups.
We start by giving a general method for proving property~\(\rt\)
for quotients obtained by replacing a normal subgroup with its
nilpotent quotients. The method applies to both families above and it uses the action on the abelianization of the normal subgroup, its first homology group,
together with an affine realization of the resulting extension.
Once this first quotient has property~\(\rt\), a permanence theorem
for nilpotent normal subgroups gives all subsequent stages.
The first of our results isolates the hypotheses needed for our
principal applications.

For a group \(K\), denote its lower central series by
\[
 K_{[1]}=K,\qquad K_{[j+1]}=[K,K_{[j]}].
\]

\begin{maintheorem}\label{thm:strategy}
Let \(G\) be a finitely generated group and \(K\subseteq G\)
a normal subgroup with finitely generated abelianization. Put
\[
 A=\bigl(K/[K,K]\bigr)/\tor\bigl(K/[K,K]\bigr),\qquad
 V=A\otimes_\Z\R,\qquad \Gamma=G/K.
\]
Suppose that:
\begin{enumerate}[label=\textup{(\roman*)}]
 \item \(\Gamma\) is a lattice in a second-countable unimodular
 locally compact group \(S\) with property~\(\rt\);
 \item the conjugation action of \(\Gamma\) on \(V\) extends to
 a continuous representation \(\rho:S\to\GL(V)\) satisfying
 \(\lvert\det\rho(s)\rvert=1\) for every \(s\in S\);
 \item \(V\) is spanned by vectors \(v\) for which there is a
 sequence \((s_j)\) in \(S\) such that \(\rho(s_j)v\to0\);
 \item some element \(z\) of the center of \(\Gamma\) fixes no
 nonzero vector of \(V\).
\end{enumerate}
Then \(G/K_{[c+1]}\) has property~\(\rt\) for every \(c\ge1\).
More generally, every quotient of \(G\) in which the image of \(K\)
is nilpotent has property~\(\rt\).
\end{maintheorem}

In the applications below the central element in~(iv) acts as \(-I\).
A lift of this element provides affine coordinates for
\(G/[K,K]\), after removing the finite torsion in its abelian
kernel.
The contraction assumption~(iii) gives property~\(\rt\) for
\(V\rtimes S\), and the affine realization identifies the
discrete group as a lattice in it. Finally, the theorem of
Chatterji, Witte Morris and Shah \cite{CWMS} implies that
\[
 G/K_{[c+1]}\text{ has property }\rt
 \quad\Longleftrightarrow\quad
 G/[K,K]\text{ has property }\rt,\qquad c\ge1.
\]
Thus property \(\rt\) for the higher stages reduces to that of the first stage and no separate analysis of the higher lower-central quotients
of \(K\) is necessary. We formulate a slightly more general affine criterion in
\cref{thm:generalcriterion}, of which \cref{thm:strategy} is a special case.

We then move on to apply the above method to obtain property \(\rt\) for several families 
arising as quotients of mapping class groups of surfaces, as well as some higher-dimensional manifolds.
The general pattern for applying \cref{thm:strategy} is via homology representations of these mapping class groups:
in all cases we consider the homomorphism 
\[ \Mod(M) \to \operatorname \Aut(H),\]
where \(H\) is an appropriate homology group of \(M\). The explicit realizations of this principle will be different for different groups, however the general mechanism remains the same. 

We first apply this method to mapping class groups of surfaces. Denote by
\(H_g=H_1(\Sigma_g;\Z)\) the first homology of the surface and by
\[
 \mathcal{T}_g
 =\ker\bigl(\Mod(\Sigma_g)\longrightarrow\Sp_{2g}(\Z)\bigr)
\]
the Torelli group. For \(c\ge1\), set
\[
 \mathcal{Q}_{g,c}=\Mod(\Sigma_g)/(\mathcal{T}_g)_{[c+1]},
 \qquad
 \mathcal{N}_{g,c}=\mathcal{T}_g/(\mathcal{T}_g)_{[c+1]}.
\]
The quotient \(\mathcal{Q}_{g,c}/\mathcal{N}_{g,c}\) is the symplectic group
\(\Sp_{2g}(\Z)\).
Johnson's abelianization theorem \cite{JohnsonIII} identifies
the torsion-free abelianization of \(\mathcal{T}_g\) with
\[
 A_g=\Lambda^3H_g/(H_g\wedge\Omega),
\]
where \(\Omega\) is the inverse symplectic form.
Its action extends to \(\Sp_{2g}(\R)\). The central element
\(-I\) acts as \(-1\), and a diagonal subgroup contracts a
spanning set. These facts verify the hypotheses of
\cref{thm:strategy} and give our first application.

\begin{maintheorem}\label{thm:main}
For every \(g\ge3\) and every integer \(c\ge1\), the group
\(\mathcal{Q}_{g,c}\) is finitely generated, infinite and has
Kazhdan's property~\(\rt\). Its normal subgroup \(\mathcal{N}_{g,c}\)
is finitely generated and nilpotent, has infinite abelianization,
and satisfies \(\mathcal{Q}_{g,c}/\mathcal{N}_{g,c}\cong\Sp_{2g}(\Z)\).
\end{maintheorem}

The lower central series quotients of the Torelli group provide successive approximations to the structure of \(\operatorname{Mod}(\Sigma_g)\) beyond its action on first homology. The groups \(\mathcal Q_{g,c}\) all possess 
the symplectic quotient while incorporating progressively deeper commutators in the Torelli group, and hence the interaction between the arithmetic quotient and its kernel. Their study belongs to the classical program developed through Johnson's homomorphisms and Hain's work on the rational lower central series and relative completion \cite{Johnson1980,JohnsonIII,Hain}. Morita extended the Johnson homomorphism to a crossed homomorphism on the mapping class group \cite{MoritaExtension}. His work on the Casson invariant relates the lower central series to the topology of homology \(3\)-spheres: the invariant detects a difference between this series and the Johnson filtration, defined by the action on nilpotent quotients of the surface group \cite{MoritaCasson}. On the rigidity side, Ershov and He proved that every finite-index subgroup of \(\mathcal Q_{g,c}\) has finite abelianization when \(g\ge8c+4\) \cite[Theorems~1.4 and~1.9]{ErshovHe}. Our result shows much stronger rigidity of these groups.

The genus assumption is essential. Indeed, we show that for every \(c\ge1\), the
group \(\mathcal{Q}_{2,c}\) has a finite-index subgroup surjecting
onto \(\Z\), and hence does not have property~\(\rt\).
The first-stage $c=1$ affine realization for \(g\ge3\) already appears
in the work of Hain and Matsumoto \cite[Lemma~6.3]{HainMatsumoto} and in \cref{sec:groups} we explain how it enters the argument.
For \(c\ge2\), these quotients are not residually finite, and therefore
non-linear; see \cref{thm:nonresidual} below.

As our second example we consider the case of \(\Aut(F_3)\), the automorphism group of the free group on \(3\) generators.
For \(\Aut(F_3)\), the relevant normal subgroup is
\[
 \operatorname{IA}_3
 =\ker\bigl(\Aut(F_3)\longrightarrow\GL_3(\Z)\bigr).
\]
For the free-group application here and in \cref{sec:freeautomorphisms},
we use the local notation \(H_n=H_1(F_n;\Z)\) and
\(H_n^*=\Hom(H_n,\Z)\).
The first Johnson homomorphism (see \cite{Johnson1980} for its Torelli-group
version) identifies the abelianization of
\(\operatorname{IA}_3\) with
\(H_3^*\otimes\Lambda^2H_3\), see \cite{SatohSurvey} for details.
Here the arithmetic quotient is \(\GL_3(\Z)\), the ambient group
is
\(\{s\in\GL_3(\R):\lvert\det s\rvert=1\}\), and again
\(-I\) acts as \(-1\) on the abelianisation of the  kernel.
The contraction condition is verified by a diagonal subgroup.
For the family of groups
\[
 \mathcal{P}_{3,c}
 =\Aut(F_3)/(\operatorname{IA}_3)_{[c+1]},
\]
we obtain the following 

\begin{maintheorem}\label{thm:autquotients}
For every integer \(c\ge1\), the group
\(\mathcal{P}_{3,c}\) is finitely generated, infinite and has
property~\(\rt\). The normal subgroup
\(\operatorname{IA}_3/(\operatorname{IA}_3)_{[c+1]}\) is
finitely generated and nilpotent, and its abelianization is the
infinite free abelian group
\[
 H_3^*\otimes\bigwedge\nolimits^2 H_3.
\]
In particular, every quotient of \(\Aut(F_3)\) in which
\(\operatorname{IA}_3\) has nilpotent image has property~\(\rt\).
\end{maintheorem}

\Cref{thm:autquotients} is stated for rank \(n=3\), although the same argument also applies to the case \(n\ge4\).
In the latter case the
analogous property~\(\rt\) assertion already follows from
property~\(\rt\) of \(\Aut(F_n)\).
Recall that \(\operatorname{Aut}(F_3)\) does not have
property~\(\rt\) \cite{McCool,GrunewaldLubotzky}.
In \cref{sec:freeautomorphisms} we discuss the relation with
Lubotzky and Pak's earlier results on tame automorphisms of free
nilpotent groups \cite{LubotzkyPak}, which also give the first
stage of this application.

An important feature of the argument in Theorem \ref{thm:strategy} is that the same method
applies to both families, \(\Mod(\Sigma_g)\) and \(\Aut(F_n)\). The homology modules and arithmetic
groups differ, but the passage from the action on first homology
to an affine lattice, and then to every nilpotent stage, is the
same. 

A number of other classes of groups, related to the ones considered in the above theorems, are discussed in \cref{sec:finalremarks}.

For the 
quotients \(\mathcal Q_{g,c}\) we additionally prove nonlinearity by showing they are not residually finite. We use Hain's comparison of unipotent and
relative completions \cite{Hain}, which identifies a rational central
line that survives at every stage \(c\ge2\). Combining
Hain--Matsumoto's relative pro-\(\ell\) comparison
\cite{HainMatsumoto} with Deligne's theorem on arithmetic
central extensions \cite{Deligne}, we show that the image of a fixed
infinite cyclic subgroup is trivial in every finite quotient.

\begin{maintheorem}\label{thm:nonresidual}
Let \(g\ge3\) and \(c\ge2\). The group \(\mathcal Q_{g,c}\)
contains an element of infinite order which maps to the identity
under every homomorphism to a finite group. In particular,
\(\mathcal Q_{g,c}\) is not residually finite.
\end{maintheorem}

In particular, these groups are nonlinear over every field, see \cref{cor:nonresidual-nonlinearity}.

The paper is organized as follows. In \cref{sec:general} we
develop the general strategy. We apply it to mapping class
groups in \cref{sec:groups} and to automorphism groups of free
groups in \cref{sec:freeautomorphisms}.
Further applications and several other examples of groups with property \(\rt\) that follow from out methods are discussed in \cref{sec:finalremarks}.
In \cref{sec:nonresidual} we show that the Torelli lower central series quotients are not residually finite. 

\section{A general strategy for nilpotent-kernel quotients}
\label{sec:general}
\label{sec:permanence}

In this section we develop a general argument for proving property~\(\rt\).
It applies to an arbitrary
countable group with a normal subgroup and consists of two parts: algebraic and analytic. 
The algebraic part reduces property \(\rt\) for
an entire family of nilpotent-kernel quotients to proving it for a single quotient,
obtained by abelianizing the kernel. The analytic part  then gives a
sufficient condition for this first quotient to have
property~\(\rt\): an affine realization over a lattice in a
locally compact group with property~\(\rt\), together with
contraction of the translation directions.
The hypotheses concerning lattices,
modules and contraction will be checked separately for every example we consider.

\subsection{Conventions and elementary permanence properties}

We use the commutator convention \([x,y]=xyx^{-1}y^{-1}\).
For subgroups, brackets denote the subgroup generated by the
corresponding commutators. The lower central series of a group
\(G\) is
\[
 G_{[1]}=G,\qquad G_{[j+1]}=[G,G_{[j]}]\quad(j\ge1).
\]
The group is \emph{nilpotent of class at most \(c\)} if
\(G_{[c+1]}=1\). The abelianization is \(G/[G,G]\).
If \(H\subseteq G\) is a normal subgroup, all its lower-central subgroups are
characteristic in \(H\), hence normal in \(G\).
For a finitely generated abelian group \(B\), write
\(\tor(B)\) for its finite torsion subgroup.

All locally compact groups below are Hausdorff. A unitary
representation \(\pi\) of such a group \(G\) on a complex
Hilbert space \(\HH\) is assumed to be strongly continuous: for each
\(\xi\in\HH\), the map \(g\mapsto\pi(g)\xi\) is norm
continuous. It has \emph{almost invariant unit vectors} if, for
every compact subset \(C\subset G\) and every \(\eps>0\),
there is \(\xi\in\HH\) with \(\|\xi\|=1\) and
\[
 \sup_{s\in C}\|\pi(s)\xi-\xi\|<\eps.
\]
For a subgroup \(H\subseteq G\), denote
\[
 \HH^H=\{\xi\in\HH:\pi(h)\xi=\xi\text{ for every }h\in H\}.
\]
The group has \emph{Kazhdan's property~\(\rt\)} if every unitary representation \(\pi\) of \(G\) with almost invariant vectors satisfies \(\HH^G\neq 0\).

The pair \((G,H)\) has \emph{relative property~\(\rt\)} if every unitary representation \(\pi\) of \(G\) with almost invariant vectors has \(\HH^H\ne0\).

We shall use the following standard permanence properties for
countable discrete groups. Property~\(\rt\) passes to quotient
groups \cite[Theorem~1.3.4]{BHV}. If \(H\subseteq G\) is a
subgroup of finite index, then \(G\) has property~\(\rt\) if
and only if \(H\) does \cite[Theorem~1.7.1]{BHV}. Property~\(\rt\)
is also invariant under passage through finite normal kernels:
if \(F\subseteq G\) is a finite normal subgroup, then \(G\)
has property~\(\rt\) if and only if \(G/F\) does
\cite[Propositions~1.1.5 and~1.7.6]{BHV}.

\begin{lemma}
\label{lem:normalextension}
Let \(N\subseteq G\) be a normal subgroup of a countable
discrete group. The following conditions are equivalent:
\begin{enumerate}[label=\textup{(\roman*)},beginpenalty=10000,midpenalty=10000]
 \item \(G\) has property~\(\rt\);
 \item \((G,N)\) has relative property~\(\rt\) and \(G/N\)
 has property~\(\rt\).
\end{enumerate}
\end{lemma}
\begin{proof}
Suppose~(i) holds. The relative assertion in~(ii) follows
from the definition, and the assertion about \(G/N\) follows
from quotient permanence.

Conversely, assume~(ii), and let \(\pi\) have almost invariant unit vectors.
Choose finite subsets \(C_n\subseteq C_{n+1}\), \(\bigcup C_n=G \), and unit
vectors \(\xi_n\) such that
\[
 \sup_{g\in C_n}\|\pi(g)\xi_n-\xi_n\|<1/n.
\]
In particular, for every fixed \(g\in G\),
\(\|\pi(g)\xi_n-\xi_n\|\to0\).
The orthogonal projection
\(P:\HH\to\HH^N\) commutes with every \(\pi(g)\).
We claim that \(\|(1-P)\xi_n\|\to0\).
Otherwise, some subsequence satisfies
\(\|(1-P)\xi_{n_j}\|\ge\delta>0\).
The  normalized vectors \( \eta_j=(1-P)\xi_{n_j}/\|(1-P)\xi_{n_j}\| \in(\HH^N)^\perp \)
are unit vectors, and we have
\[
 \begin{aligned}
 \|\pi(g)\eta_j-\eta_j\| \le\delta^{-1}\|\pi(g)\xi_{n_j}-\xi_{n_j}\|.
 \end{aligned}
\]
The right-hand side tends to zero uniformly on finite sets, hence the representation on
\((\HH^N)^\perp\) has almost invariant unit vectors.
Since it has no nonzero \(N\)-fixed vector, this contradicts
relative property~\(\rt\).

We have  \(\|P\xi_n\|^2=1-\|(1-P)\xi_n\|^2\to1\).
For all sufficiently large \(n\), the vectors
\(P\xi_n/\|P\xi_n\|\) are defined, and the same estimate,
with \(P\) in place of \(1-P\), shows that they are almost
invariant. The representation on \(\HH^N\) factors through
\(G/N\). Property~\(\rt\) of this quotient
guarantees the existence of a non-zero vector, which is fixed by the whole group \(G\).
\end{proof}

\subsection{Nilpotent normal subgroups and quotient towers}

The following result is the starting point for our investigation.
\begin{theorem}[Chatterji--Witte Morris--Shah]
\label{thm:cwms}
Let \(N\subseteq G\) be a nilpotent normal subgroup of a countable
discrete group. The following conditions are equivalent:
\begin{enumerate}[label=\textup{(\roman*)},beginpenalty=10000,midpenalty=10000]
 \item \((G,N)\) has relative property~\(\rt\);
 \item \(\bigl(G/[N,N],\,N/[N,N]\bigr)\) has relative
 property~\(\rt\).
\end{enumerate}
\end{theorem}

This result is \cite[Theorem~1.2]{CWMS} in the special case of discrete groups. 

\begin{theorem}
\label{thm:nilpotentpermanence}
Let \(N\subseteq G\) be a nilpotent normal subgroup, with \(G\)
countable and discrete. The following conditions are equivalent:
\begin{enumerate}[label=\textup{(\roman*)},beginpenalty=10000,midpenalty=10000]
 \item \(G\) has property~\(\rt\);
 \item \(G/[N,N]\) has property~\(\rt\).
\end{enumerate}
\end{theorem}
\begin{proof}
The implication \(\textup{(i)}\Rightarrow\textup{(ii)}\) follows since property~\(\rt\) is preserved by quotients. 
Conversely, assume~(ii). Then \(G/[N,N]\) has relative property~\(\rt\)
with respect to every subgroup, in particular \(N/[N,N]\).
By \cref{thm:cwms}, the pair \((G,N)\) has relative property~\(\rt\).
Also \(G/N\), as a quotient of \(G/[N,N]\), has
property~\(\rt\). The claim follows from \cref{lem:normalextension}.
\end{proof}

\begin{remark}
\label{rem:nilpotenceessential}
The hypothesis cannot be replaced by virtual nilpotence or
solvability. Indeed, let
\[
 G=N=D_\infty
 =\langle r,s\mid s^2=1,\ srs^{-1}=r^{-1}\rangle
 \cong\Z\rtimes C_2.
\]
Then \([N,N]=\langle r^2\rangle\) and
\(G/[N,N]\cong C_2\times C_2\) is finite, hence has
property~\(\rt\). However, \(\langle r\rangle\cong\Z\)
has index two in \(G\), so \(G\) does not have
property~\(\rt\).
This subgroup \(N\) is virtually cyclic, virtually nilpotent,
metabelian and polycyclic. Thus the same example excludes each
of these properties, as well as solvability. 
\end{remark}

\begin{remark}
\label{rem:nilpotentextensions}
The conclusion of \cref{thm:nilpotentpermanence} remains valid
if there is a normal subgroup \(K\subseteq G\), with \(K\subseteq N\), such that
\((G,K)\) has relative property~\(\rt\) and \(N/K\) is
nilpotent. To see the reverse implication, observe that
\[
 (G/K)/[N/K,N/K]\cong G/(K[N,N])
\]
is a quotient of \(G/[N,N]\). Applying \cref{thm:nilpotentpermanence}
to the nilpotent normal subgroup \(N/K\subseteq G/K\) gives
property~\(\rt\) for \(G/K\). Relative property~\(\rt\) for \(K\)
then gives property~\(\rt\) for \(G\).
It suffices, in particular, that \(K\) itself has
property~\(\rt\) and \(N/K\) is nilpotent.

Another special case is when \(N\) is finite-by-nilpotent; that is,
there is a finite normal subgroup \(F\subseteq N\) with
\(N/F\) nilpotent. If its nilpotency class is at most \(c\),
then \(K=N_{[c+1]}\subseteq F\) is finite and characteristic
in \(N\), hence it is normal in \(G\). To obtain the claim apply the preceding argument
to this \(K\).
\end{remark}

The next statement shows that property~\(\rt\) for the first quotient determines it for the whole tower of quotients.
\begin{theorem}
\label{thm:quotienttower}
Let \(K\subseteq G\) be a normal subgroup, where \(G\) is countable
and discrete. For every \(c\ge1\), the following conditions are equivalent:
\begin{enumerate}[label=\textup{(\roman*)},beginpenalty=10000,midpenalty=10000]
 \item \(G/K_{[c+1]}\) has property~\(\rt\);
 \item \(G/[K,K]\) has property~\(\rt\).
\end{enumerate}
If these conditions hold, every quotient
of \(G\) in which the image of \(K\) is nilpotent has
property~\(\rt\). The same conclusion holds if that image
is finite-by-nilpotent.
\end{theorem}
\begin{proof}
The subgroup \(K/K_{[c+1]}\) is nilpotent and normal in
\(G/K_{[c+1]}\), and
\[
 \begin{aligned}
 [K/K_{[c+1]},K/K_{[c+1]}]&=[K,K]/K_{[c+1]},\\
 (G/K_{[c+1]})/\bigl([K,K]/K_{[c+1]}\bigr)&\cong G/[K,K].
 \end{aligned}
\]
The equivalence is therefore \cref{thm:nilpotentpermanence}.

More generally, let \(p:G\rightarrow E\) be a surjective homomorphism, and write
\(N=p(K)\). There is a surjection
\[
 G/[K,K]\rightarrow E/[N,N],
\]
because \(p([K,K])=[N,N]\). Thus \(E/[N,N]\) has
property~\(\rt\). If \(N\) is nilpotent, apply
\cref{thm:nilpotentpermanence}, and if it is finite-by-nilpotent,
apply \cref{rem:nilpotentextensions}.
\end{proof}

\subsection{Contracting affine groups}
\label{sec:generalaffine}

We will now describe a contraction argument, that will be crucial for proving property \(\rt\).
Let \(V\) be a finite-dimensional real vector space and
\(\rho:S\to\GL(V)\) a continuous representation of a
locally compact group. The \emph{semidirect product}
\(V\rtimes_\rho S\) has underlying space \(V\times S\) and group operation defined by
\[
 (v,s)(w,t)=(v+\rho(s)w,st),\qquad
 (v,s)^{-1}=(-\rho(s^{-1})v,s^{-1}).
\]
We omit the subscript when the action is understood. The vector
subgroup \(V\times\{1\}\) consists of translations, and
\(\{0\}\times S\) is identified with \(S\).

\begin{lemma}
\label{lem:contraction}
Given a continuous unitary representation \(\pi\) of
\(V\rtimes S\), suppose that \(\xi\) is fixed by \(S\).
If \(v\in V\) and there exists  sequence \(s_n\in S\), satisfying
\(\rho(s_n)v\to0\), then \(\xi\) is fixed by the
translation \((v,1)\).
\end{lemma}
\begin{proof}
Conjugation by \((0,s_n)\) sends \((v,1)\) to
\((\rho(s_n)v,1)\). Since \(\xi\) is fixed by
\((0,s_n)\) and its inverse, unitarity gives
\[
 \|\pi(v,1)\xi-\xi\|
 =\|\pi(\rho(s_n)v,1)\xi-\xi\|\longrightarrow0.
\]
The convergence follows from strong continuity at the identity of the
translation subgroup. The expression on the left is independent
of \(n\), and is therefore zero.
\end{proof}

The above lemma is an elementary form of the Mautner phenomenon,
see \cite[Lemma~1.4.8]{BHV}.

\begin{proposition}
\label{prop:affineT}
Suppose \(S\) has property~\(\rt\). If \(V\) is spanned
by vectors each of which can be contracted to zero, in the sense of Lemma \ref{lem:contraction}, by a sequence
in \(S\), then \(V\rtimes S\) has property~\(\rt\).
\end{proposition}
\begin{proof}
Let a unitary representation \(\pi\) have almost invariant unit vectors. Its restriction
to \(S\) has almost invariant vectors, because a compact subset
of \(S\) is compact in \(V\rtimes S\). Property~\(\rt\)
of \(S\) gives a nonzero vector \(\xi\) fixed by \(S\).
By \cref{lem:contraction}, it is fixed by every contractible
translation. A sequence contracting \(v\) also contracts every
real multiple of \(v\). Every vector in \(V\) is a finite
sum of these multiples, and translations add under multiplication.
Thus \(\xi\) is fixed by all of \(V\), and hence by
\(V\rtimes S\).
\end{proof}

We next explain the passage to discrete groups. 
A \emph{left Haar measure} on a
locally compact group is a nonzero left-translation-invariant
Radon measure. Such a measure exists and is unique up to a positive scalar.
A group is \emph{unimodular} if its left Haar measure is also
right invariant. A discrete subgroup \(\Gamma\) of a locally
compact group \(S\) is a \emph{lattice} if \(S/\Gamma\)
admits a nonzero finite \(S\)-invariant Radon measure.
Here \(S/\Gamma\) denotes cosets \(s\Gamma\), obtained
by the right multiplication action of \(\Gamma\) on \(S\).
For a second-countable unimodular group, the lattice condition
can equivalently be expressed by a Borel fundamental domain
of positive finite Haar measure, see
\cite[Sections~A.3 and~B.1--B.2]{BHV}.

A \emph{full lattice} in \(V\) is a subgroup
\(A=\Z b_1+\cdots+\Z b_d\), where \(b_1,\ldots,b_d\)
is a real basis of \(V\). Its quotient \(V/A\) is a compact
torus. We allow \(d=0\), in which case both \(A\) and
\(V\) are zero and their fundamental domain is a point of
measure one.
The next statement gives sufficient conditions for a discrete subgroup of the affine group to be a lattice.
\begin{proposition}
\label{prop:affinelattice}
Let \(S\) be a second-countable unimodular locally compact
group, \(V\) a finite-dimensional real vector space, and
\(\rho:S\to\GL(V)\) a continuous representation preserving
Lebesgue measure, that is,
\[
 |\det\rho(s)|=1\qquad(s\in S).
\]
Let \(p:V\rtimes S\to S\) be the coordinate projection,
and identify \(V\) with the translation subgroup
\(V\times\{1\}\). Suppose that a subgroup
\(\Delta\subseteq V\rtimes S\) satisfies the following:
\begin{enumerate}[label=\textup{(\roman*)},beginpenalty=10000,midpenalty=10000]
 \item \(\Gamma=p(\Delta)\) is a lattice in \(S\);
 \item \(A=\Delta\cap V\) is a full lattice in \(V\).
\end{enumerate}
Then \(\Delta\) is a lattice in \(V\rtimes S\).
In particular, \(A\rtimes\Gamma\) is a
lattice whenever \(\Gamma\) is a lattice in \(S\) preserving
a full lattice \(A\) in \(V\).

If, moreover, \(S\) has property~\(\rt\) and \(V\) is
spanned by vectors each of which can be contracted to zero
by a sequence in \(S\), then both \(V\rtimes S\) and
\(\Delta\) have property~\(\rt\).
\end{proposition}
\begin{proof}
Let \(L=V\rtimes S\). Discreteness of \(\Gamma\) and \(A\)
implies that a sufficiently small product neighborhood of the
identity meets \(\Delta\) only at the identity. Hence
\(\Delta\) is a closed discrete subgroup.
Conjugation in \(\Delta\) also gives \(\rho(\Gamma)A=A\).

Let \(\nu\) be Haar measure on \(S\). Choose a half-open
fundamental parallelepiped \(P\) for \(A\), and normalize
Lebesgue measure \(\lambda\) by \(\lambda(P)=1\).
The determinant assumption and unimodularity of \(S\)
make \(\mu=\lambda\times\nu\) a left and right Haar measure
on \(L\).

Choose a Borel fundamental domain \(\mathcal D\) for
\(S/\Gamma\), with \(0<\nu(\mathcal D)<\infty\), and put
\[
 \mathcal F=\{(\rho(s)u,s):s\in\mathcal D,\ u\in P\}.
\]
This is a Borel fundamental domain for \(L/\Delta\).
Indeed, choose a lift \((b_\gamma,\gamma)\in\Delta\) for
each \(\gamma\in\Gamma\). Given \((v,t)\in L\), write
uniquely \(t=s\gamma\), with \(s\in\mathcal D\), and
\(\rho(s)^{-1}v-b_\gamma=u+a\), with \(u\in P\), \(a\in A\).
Then
\[
 (v,t)=(\rho(s)u,s)(a+b_\gamma,\gamma),
 \qquad (a+b_\gamma,\gamma)\in\Delta.
\]
These unique decompositions give exactly one representative
in \(\mathcal F\) for each coset, without requiring a
homomorphic choice of lifts.

Volume preservation now gives
\[
 0<\mu(\mathcal F)
 =\int_{\mathcal D}\lambda(\rho(s)P)\,d\nu(s)
 =\nu(\mathcal D)<\infty.
\]
Thus \(\Delta\) is a lattice by
\cite[Proposition~B.2.4]{BHV}. Equivalently, restricting
\(\mu\) to \(\mathcal F\) and passing to \(L/\Delta\)
gives its finite invariant quotient measure, see
\cite[Corollary~B.1.7 and Proposition~B.2.4]{BHV}.

Finally, under the additional hypotheses, \(L\) has
property~\(\rt\) by \cref{prop:affineT}, and so does its
lattice \(\Delta\) by \cite[Theorem~1.7.1]{BHV}.
\end{proof}

\subsection{The main criterion for property \(\rt\)}

Consider an exact sequence
\begin{equation}\label{eq:extensionaction}
 1\longrightarrow A\longrightarrow E
 \xrightarrow{p}\Gamma\longrightarrow1
\end{equation}
with \(A\) free abelian of finite rank. We will use additive notation for the group operation in \(A\). 
Conjugation on \(A\)
factors through \(\Gamma\), since conjugation by \(A\)
is trivial on \(A\), inducing a map
\[ \rho: \Gamma \to \operatorname{Aut}(A). \]
A map \(b:E\to V\),
where \(V=A\otimes\R\), is a \emph{crossed homomorphism}
if
\[
 b(xy)=b(x)+\rho(p(x))b(y).
\]
This identity says exactly that
\(x\mapsto(b(x),p(x))\) is a homomorphism into
\(V\rtimes\Gamma\). If \(b|_A\) is the inclusion of
\(A\) in \(V\), that homomorphism is injective: an element
of its kernel lies in \(A\), where its first coordinate is
the element itself.

The following elementary construction often provides such a crossed homomorphism.
Its usefulness lies in the fact that an actual splitting of \eqref{eq:extensionaction} is not required.

\begin{proposition}\label{prop:centralaffine}
Consider the extension \eqref{eq:extensionaction} and suppose there is \(z\in Z(\Gamma)\)
such that
\[
 D=I-\rho(z):V\longrightarrow V
\]
is invertible. Then there is an injective homomorphism
\[
 \phi:E\longrightarrow A\rtimes\Gamma
\]
such that
\[
 \operatorname{pr}_{\Gamma}\circ\phi=p,
 \qquad \phi(a)=(Da,1)\quad(a\in A),
\]
where \(\operatorname{pr}_{\Gamma}:A\rtimes\Gamma\to\Gamma\)
is the coordinate projection \((a,\gamma)\mapsto\gamma\).
The image of \(\phi\) has index
\[
 [A:D(A)]=|\det D|.
\]
There is also an injective homomorphism into
\(V\rtimes\Gamma\) which restricts to the identity on \(A\).
In particular, if \(z\) acts as \(-1\) on \(A\), the first
index is \(2^{\operatorname{rank}A}\).
\end{proposition}
\begin{proof}
Choose a lift \(u\in E\) of \(z\), and put
\(d(x)=[x,u]\). Since \(z\) is central in \(\Gamma\),
the image of this commutator in \(\Gamma\) is trivial.
Hence \(d(x)\in A\). The identity
\[
 [xy,u]=x[y,u]x^{-1}[x,u]
\]
and commutativity in \(A\) imply
\[
 d(xy)=d(x)+\rho(p(x))d(y).
\]
For \(a\in A\), conjugation by \(u\) is \(\rho(z)\), so
\[
 d(a)=a-\rho(z)a=Da.
\]
It follows that \(\phi(x)=(d(x),p(x))\) is a homomorphism.
If \(\phi(x)=1\), then \(p(x)=1\), so \(x=a\in A\),
and \(Da=0\). Invertibility of \(D\) gives \(a=0\).

The image of \(\phi\) projects onto \(\Gamma\), and its
intersection with the translation subgroup is \(D(A)\).
Every coset of this image in \(A\rtimes\Gamma\) has a
representative in \(A\): multiply a given element by the
inverse of an image element having the same second coordinate.
Two translation representatives give the same coset exactly
when their difference belongs to \(D(A)\).
Therefore the index is \([A:D(A)]\).
The operator \(D\) has an integral matrix in a basis of \(A\)
and nonzero determinant. The usual lattice index formula gives
\([A:D(A)]=|\det D|\).

Because \(z\) is central, \(D\) commutes with every
\(\rho(\gamma)\), as does \(D^{-1}\). Thus
\[
 b(x)=D^{-1}d(x)
\]
is a crossed homomorphism into \(V\), with \(b(a)=a\).
The resulting normalized affine embedding has translation
kernel exactly \(A\). If \(\rho(z)=-I\), then \(D=2I\),
which gives the final assertion.
\end{proof}

The normalized coordinates may lie in \(D^{-1}A\) rather
than \(A\). If \(D=2I\), they can in fact be half-integral.

We can now combine the above results into a strategy for proving property~\(\rt\). 
Theorem \ref{thm:strategy} stated in the introduction is a slightly specialized case of the following theorem.
\begin{theorem}
\label{thm:generalcriterion}
Let \(G\) be a countable discrete group and \(K\subseteq G\)
a normal subgroup. Suppose \(B=K/[K,K]\) is finitely generated abelian.
Denote
\[
 A=B/\tor(B),\quad V=A\otimes_\Z\R,
 \quad\Gamma=G/K,
 \quad E=(G/[K,K])/\tor(B).
\]
The subgroup \(\tor(B)\) is normal in \(G/[K,K]\), and there is
an exact sequence \[1\to A\to E\to\Gamma\to1.\]
Assume the following:
\begin{enumerate}[label=\textup{(\roman*)}]
 \item \(\Gamma\) is a lattice in a second-countable
 unimodular locally compact group \(S\) with property~\(\rt\).
 \item The action of \(\Gamma\) on \(A\), extended linearly
 to \(V\), is the restriction of a continuous representation
 \(\rho:S\to\GL(V)\) satisfying
 \(|\det\rho(s)|=1\) for every \(s\in S\).
 \item The vector space \(V\) is spanned by vectors each of
 which can be contracted to zero by a sequence in \(S\).
 \item There is a crossed homomorphism \(b:E\to V\) whose
 restriction to \(A\) is its natural inclusion in \(V\).
\end{enumerate}
Then \(G/[K,K]\) has property~\(\rt\). Consequently
\(G/K_{[c+1]}\) has property~\(\rt\) for every \(c\ge1\),
as does every quotient of \(G\) with nilpotent or
finite-by-nilpotent image of \(K\).

In condition \textup{(iv)}, it suffices that there be
\(z\in Z(\Gamma)\) for which \(I-\rho(z)\) is invertible
on \(V\). In particular, it suffices that a central element
act as \(-1\).
\end{theorem}
\begin{proof}
Conjugation preserves torsion in the abelian normal subgroup
\(B\), so \(\tor(B)\) is finite and normal as asserted.
The crossed homomorphism in~(iv) embeds \(E\) into
\(V\rtimes S\), with projection \(\Gamma\) and translation
kernel exactly \(A\). Conditions~(i)--(iii) and
\cref{prop:affinelattice} therefore give property~\(\rt\)
for \(E\). The natural surjection
\(G/[K,K]\to E\) has finite kernel \(\tor(B)\),
so permanence under finite normal kernels gives property~\(\rt\) for
\(G/[K,K]\). Apply \cref{thm:quotienttower} for the remaining
statements. The sufficient condition involving a central element is
\cref{prop:centralaffine}.
\end{proof}

We remark that condition~(iv) can also be expressed in terms
of the extension class. Choose a section
\(s:\Gamma\to E\), \(s(1)=1\), and write
\[
 s(\gamma)s(\eta)
 =\alpha(\gamma,\eta)s(\gamma\eta),\qquad
 \alpha(\gamma,\eta)\in A.
\]
If a map \(h:\Gamma\to V\) satisfies
\[
 \alpha(\gamma,\eta)
 =h(\gamma)+\rho(\gamma)h(\eta)-h(\gamma\eta),
\]
then \(b(a s(\gamma))=a+h(\gamma)\) is the required
crossed homomorphism. Conversely, restricting such a crossed homomorphism \(b\)
to the section gives the above formula. Thus~(iv) says precisely
that the extension class becomes zero in \(H^2(\Gamma;V)\).
The central-element construction is an explicit way of obtaining
this vanishing without calculating a cohomology group.

\section{The Torelli quotients}
\label{sec:groups}

\subsection{The Torelli group and its lower central series}

All mapping classes are represented by orientation-preserving
homeomorphisms. Denote \[H_g=H_1(\Sigma_g;\Z).\] The algebraic
intersection pairing \(\omega:H_g\times H_g\to\Z\) is
unimodular and alternating. Choose a symplectic basis
\(e_1,\ldots,e_g,f_1,\ldots,f_g\), so that
\(\omega(e_i,f_j)=\delta_{ij}\) and the pairings among the
\(e_i\)'s and among the \(f_i\)'s vanish. Denote the algebraic
symplectic group by \(\Sp_{2g}\).
The homological action is a surjective homomorphism
\[
 \sigma:\Mod(\Sigma_g)\longrightarrow\Sp_{2g}(\Z).
\]
Its kernel is the Torelli group \(\mathcal{T}_g\). We refer to  \cite{FarbMargalit} for this and
other basic facts about mapping class groups. 

For \(c\ge1\), we write
\[
 \mathcal{Q}_{g,c}=\Mod(\Sigma_g)/(\mathcal{T}_g)_{[c+1]},
 \qquad
 \mathcal{N}_{g,c}=\mathcal{T}_g/(\mathcal{T}_g)_{[c+1]}.
\]
We assume \(g\ge3\) throughout. 
The subgroup \((\mathcal{T}_g)_{[j]}\) is characteristic in \(\mathcal{T}_g\), and thus it is
normal in \(\Mod(\Sigma_g)\). Consequently, the quotient \(\mathcal{Q}_{g,c}\) is therefore
defined without lifting the symplectic action to \(\mathcal{T}_g\).
A group is
\emph{nilpotent of class at most \(c\)} when its
\((c+1)\)-st lower-central subgroup is trivial. In particular,
\(\mathcal{N}_{g,c}\) has class at most \(c\).

\begin{lemma}
\label{lem:nilab}
For \(g\ge3\) and \(c\ge1\), \(\mathcal{N}_{g,c}\) is finitely
generated and nilpotent. Moreover,
\[
 (\mathcal{N}_{g,c})_{\ab}\cong \mathcal{T}_g/[\mathcal{T}_g,\mathcal{T}_g],
 \qquad \mathcal{Q}_{g,c}/[\mathcal{N}_{g,c},\mathcal{N}_{g,c}]\cong \mathcal{Q}_{g,1}.
\]
These identifications respect the conjugation actions induced
from \(\Mod(\Sigma_g)\).
\end{lemma}
\begin{proof}
The Torelli group is finitely generated for \(g\ge3\), by
\cite{JohnsonI}, hence \(\mathcal{N}_{g,c}\) is finitely generated.
For a normal subgroup \(K\subseteq I\), the commutator
subgroup of \(I/K\) is \([I,I]K/K\). Here
\(K=I_{[c+1]}\subseteq[I,I]\), with equality when
\(c=1\). Consequently
\[
 [\mathcal{N}_{g,c},\mathcal{N}_{g,c}]=(\mathcal{T}_g)_{[2]}/(\mathcal{T}_g)_{[c+1]}.
\]
Taking quotients gives the two isomorphisms. Both are induced by
quotient maps from \(\Mod(\Sigma_g)\), so they respect the conjugation action.
\end{proof}

\subsection{The integral Johnson module}

Let
\[
 A_g=\frac{\mathcal{T}_g/[\mathcal{T}_g,\mathcal{T}_g]}
 {\tor(\mathcal{T}_g/[\mathcal{T}_g,\mathcal{T}_g])}
\]
be the torsion-free quotient of \(\mathcal{T}_g/[\mathcal{T}_g,\mathcal{T}_g]\).
Conjugation on \(\mathcal{T}_g/[\mathcal{T}_g,\mathcal{T}_g]\) is trivial for elements of \(\mathcal{T}_g\),
so it factors through \(\Sp_{2g}(\Z)\). The same is true on
\(\tor(\mathcal{T}_g/[\mathcal{T}_g,\mathcal{T}_g])\) and \(A_g\).

Let \[\Omega=\sum_{i=1}^g e_i\wedge f_i\] and define a map \(j:H_g\to\Lambda^3 H_g\) by setting 
\[j(v)=v\wedge\Omega.\] The map \(j\) is equivariant.

The following is the torsion-free part of Johnson's theorem
\cite{JohnsonIII}, see also \cite[p.~398 and Theorem~3.1]{DimcaPapadima}.
\begin{theorem}[Johnson's abelianization theorem]
\label{thm:johnsoninput}
For \(g\ge3\), the Johnson homomorphism gives a surjective
\(\Mod(\Sigma_g)\)-equivariant homomorphism
\[
 \tau_J:\mathcal{T}_g\longrightarrow\Lambda^3 H_g/j(H_g)
\]
and induces an equivariant isomorphism
\[
 A_g\cong\Lambda^3 H_g/j(H_g).
\]
\end{theorem}

The distinction between \(\mathcal{T}_g/[\mathcal{T}_g,\mathcal{T}_g]\) and \(A_g\) will be used below:
the affine realization is defined after quotienting by \(\tor(\mathcal{T}_g/[\mathcal{T}_g,\mathcal{T}_g])\),
and property~\(\rt\) then passes back through this finite kernel.

\begin{lemma}\label{lem:lattice}
For \(g\ge2\), \(j\) is injective with primitive image.
Consequently \(A_g\) is a free abelian group of rank
\[
 d_g=\binom{2g}{3}-2g \qquad(g\ge3).
\]
\end{lemma}
\begin{proof}
For a basis vector \(v=e_i\) or \(f_i\), the nonzero terms
of \(j(v)\) are \(v\wedge e_k\wedge f_k\), \(k\ne i\),
each with coefficient \(1\) or \(-1\) in the ordered exterior
basis. Such a triple has a unique full symplectic pair, and its
remaining singleton determines \(v\). Thus the supports of the
\(2g\) vectors \(j(v)\) are disjoint. Choose one term for
each \(v\), which is possible since \(g\ge2\), and take its signed
coefficient. These coefficient functionals define an integral map
\(p:\Lambda^3 H_g\to H_g\) with \(pj=1\). Hence \(j(H_g)\)
is an integral direct summand. The quotient is free and its rank
is the difference of the ranks.
\end{proof}

For the real representation we will use an elementary weight
calculation. We first specify the
meaning of weights in this argument.

\subsection{Weights for a diagonal one-parameter subgroup.}
Assume \(g\ge3\), denote \(H_{g,\R}=H_g\otimes_\Z\R\), and extend \(j\) linearly
to \(H_{g,\R}\). Johnson's identification gives
\[
 V_g:=A_g\otimes_\Z\R
 \cong \Lambda^3 H_{g,\R}/j(H_{g,\R}).
\]
For \(t\in\R\), consider the linear transformation
\[
 a_t=
 \begin{pmatrix}
 e^t I_g&0\\
 0&e^{-t}I_g
 \end{pmatrix},
 \qquad
 a_t e_i=e^t e_i,\quad a_t f_i=e^{-t}f_i.
\]
It preserves the symplectic pairing, since
\[\omega(a_t e_i,a_t f_j)=e^t e^{-t}\delta_{ij}=\delta_{ij},\]
and the other pairings remain zero. Thus
\(a_t\in\Sp_{2g}(\R)\), and \[a_{t+s}=a_ta_s,\] 
providing a one-parameter subgroup that we need for our argument.
We use the same notation for the induced operators on exterior
powers and on their invariant quotients.

For such a representation \(V\) and \(k\in\R\), define
\[
 V(k)=\{v\in V:a_t v=e^{kt}v\text{ for every }t\in\R\}.
\]
If \(V(k)\ne0\), we call \(k\) a \emph{weight} and \(V(k)\)
its \emph{weight space}, always relative to this particular
one-parameter subgroup. Thus a weight space is a simultaneous
eigenspace for all the \(a_t\),  \(k\) is the exponent, whereas
the eigenvalue of \(a_t\) is \(e^{kt}\). For example, \(e_i\)
has weight \(1\), \(f_i\) has weight \(-1\), and weight zero
means being fixed by every \(a_t\).

\begin{lemma}
\label{lem:weights}
For \(g\ge3\), the induced action of \((a_t)_{t\in\R}\)
on \(V_g\) has the direct-sum decomposition
\[
 V_g=V_g(3)\oplus V_g(1)\oplus V_g(-1)\oplus V_g(-3).
\]
On \(V_g(k)\) it is multiplication by \(e^{kt}\), and
\(V_g(0)=0\). Moreover,
\(-I\in\Sp_{2g}(\Z)\) acts as \(-1\) on \(A_g\).
\end{lemma}
\begin{proof}
Set
\[
 E=\operatorname{span}_\R\{e_1,\ldots,e_g\},\qquad
 F=\operatorname{span}_\R\{f_1,\ldots,f_g\}.
\]
Then \(H_{g,\R}=E\oplus F\), and
\[
 W:=\Lambda^3 H_{g,\R}
 =\bigoplus_{r=0}^3
   \bigl(\Lambda^r E\wedge\Lambda^{3-r}F\bigr).
\]
The summand indexed by \(r\) is spanned by exterior triples
with \(r\) factors from \(E\) and \(3-r\) factors from \(F\).
The action on an exterior product is obtained by acting on
each factor. Consequently, every vector \(w\) in this summand
satisfies
\[
 a_t w=e^{rt}e^{-(3-r)t}w=e^{(2r-3)t}w.
\]
In other words, weights add under exterior products. The four
possibilities \(r=3,2,1,0\) give weights \(3,1,-1,-3\),
respectively.

We now check explicitly what happens upon taking the quotient
by \(j(H_{g,\R})\). Since
\[
 a_t\Omega
 =\sum_{i=1}^g(e^te_i)\wedge(e^{-t}f_i)
 =\Omega,
\]
the bivector \(\Omega\) has weight zero. Therefore
\[
 a_t j(v)=a_t(v\wedge\Omega)
         =(a_t v)\wedge\Omega=j(a_t v).
\]
It follows that \(j(E)\subset W(1)\) and
\(j(F)\subset W(-1)\). Moreover,
\[
 j(H_{g,\R})=j(E)\oplus j(F),
\]
because \(j\) is injective and \(H_{g,\R}=E\oplus F\).
Thus this subspace is already the direct sum of its intersections
with the weight spaces of \(W\). Taking the quotient gives the
following natural identifications:
\begin{center}
\begin{tabular}{ccl}
\toprule
Weight \(k\) & \(W(k)\) & \(V_g(k)\), naturally isomorphic to\\
\midrule
\(3\) & \(\Lambda^3E\) & \(\Lambda^3E\)\\
\(1\) & \(\Lambda^2E\wedge F\) &
  \((\Lambda^2E\wedge F)/j(E)\)\\
\(-1\) & \(E\wedge\Lambda^2F\) &
  \((E\wedge\Lambda^2F)/j(F)\)\\
\(-3\) & \(\Lambda^3F\) & \(\Lambda^3F\)\\
\bottomrule
\end{tabular}
\end{center}
The quotient therefore inherits a direct-sum decomposition into
these weight spaces and no weight-zero space appears.
Finally, \(-I\) acts by \((-1)^3=-1\) on \(\Lambda^3H_g\),
and the same holds on its quotient \(A_g\).
\end{proof}

For example, when \(g=3\), these four weight spaces have
dimensions \(1,6,6,1\). Indeed, the two pure exterior cubes
have dimension \(\binom33=1\), while each mixed summand
has dimension \(\binom32\cdot3=9\), from which the
three-dimensional subspace \(j(E)\) or \(j(F)\) is removed.
Their dimensions add to \(\dim V_3=14\).

The absence of weight zero is the point needed later.
If \(v\in V_g(k)\), then for any norm on \(V_g\),
\[
 \|a_t v\|=e^{kt}\|v\|.
\]
Hence a vector of positive weight is contracted to zero as
\(t\to-\infty\), and a vector of negative weight is contracted
as \(t\to+\infty\). Every vector is a sum of these weight
components, so \(V_g\) is spanned by vectors that can be
contracted by elements of \(\Sp_{2g}(\R)\).

\subsection{The affine symplectic group}
\label{sec:affine}

We verify the hypotheses of the general strategy for the integral
Johnson module. The first issue is contraction and the second is
volume preservation, which is needed to pass to a lattice.

\begin{proposition}
\label{prop:arithmeticT}
For \(g\ge3\), the groups
\[
 V_g\rtimes\Sp_{2g}(\R)
 \quad\text{and}\quad A_g\rtimes\Sp_{2g}(\Z)
\]
have property~\(\rt\).
\end{proposition}
\begin{proof}
Let \(\rho:\Sp_{2g}(\R)\to\GL(V_g)\) denote the action.
The real group \(\Sp_{2g}(\R)\) has property~\(\rt\),
and \(\Sp_{2g}(\Z)\) is a lattice in it.
These are the classical higher-rank and arithmetic-lattice
theorems, see \cite[Chapters~1 and~3 and Appendix~B]{BHV}
and \cite{Raghunathan}.

By \cref{lem:weights}, \(V_g\) is the direct sum of its weight
spaces with weights \(3,1,-1,-3\). A vector of positive weight
\(k\) is contracted by \(a_{-n}\), since
\[
 \rho(a_{-n})v=e^{-kn}v\longrightarrow0.
\]
A vector of negative weight is contracted by \(a_n\). Thus
\cref{prop:affineT} proves property~\(\rt\) for
\(V_g\rtimes\Sp_{2g}(\R)\).

We next verify that volume is preserved, so that the lattice
criterion in \cref{prop:affinelattice} applies.
We verify this using determinants. 
A symplectic transformation \(s\) of \(H_{g,\R}\)
preserves the nonzero top-degree form \(\omega^{\wedge g}\).
Consequently
\[
 \det(s|_{H_{g,\R}})=1.
\]
For an invertible operator \(u\) on an \(n\)-dimensional vector
space,
\[
 \det(\Lambda^3u)=(\det u)^{\binom{n-1}{2}}.
\]
Indeed, after triangularizing over \(\C\), the diagonal entries
of \(\Lambda^3u\) are the products of triples of distinct
diagonal entries of \(u\). Each entry occurs in exactly
\(\binom{n-1}{2}\) such products. Applying this with \(n=2g\)
gives \(\det(\Lambda^3s)=1\).

The map \(j:H_{g,\R}\to\Lambda^3H_{g,\R}\) is injective and
equivariant, so the action on its image also has determinant
one. In a basis beginning with a basis of this invariant
subspace, the action on \(\Lambda^3H_{g,\R}\) is block
triangular. Its determinant is the product of the determinants
on the subspace and on the quotient. Since that quotient is
\(V_g\), we obtain
\[
 \det\rho(s)
 =\frac{\det(\Lambda^3s)}
        {\det((\Lambda^3s)|_{j(H_{g,\R})})}
 =1.
\]
Thus any Lebesgue measure \(\lambda\) on \(V_g\) satisfies
\[
 \lambda(\rho(s)C)
 =|\det\rho(s)|\,\lambda(C)=\lambda(C)
\]
for every Borel set \(C\subset V_g\). This equality is the required condition, needed in \cref{prop:affinelattice}.

The group \(\Sp_{2g}(\R)\) is unimodular. This can also be seen
by a determinant calculation: its adjoint representation on
the symplectic Lie algebra identifies with the action on
\(\operatorname{Sym}^2H_{g,\R}\), and
\[
 \det(\operatorname{Sym}^2s)=(\det s)^{2g+1}=1.
\]
Right translation by \(s\) multiplies a left-invariant volume
form by \(\det\Ad(s)^{-1}\), hence a left Haar measure
\(\nu\) on \(\Sp_{2g}(\R)\) is also right invariant.

For completeness, the finite-volume domain of
\cref{prop:affinelattice} has a concrete form here. Choose a
half-open parallelepiped \(P\) for \(A_g\) in \(V_g\), normalize
Lebesgue measure by \(\lambda(P)=1\), and choose a Borel
fundamental domain \(\mathcal D\) for the right action of
\(\Sp_{2g}(\Z)\) on \(\Sp_{2g}(\R)\), with
\(0<\nu(\mathcal D)<\infty\).
Then
\[
 \mathcal F=\{(\rho(s)p,s):p\in P,\ s\in\mathcal D\}
\]
is a fundamental domain for the right action of
\(A_g\rtimes\Sp_{2g}(\Z)\) on
\(V_g\rtimes\Sp_{2g}(\R)\). Each fiber
\(\rho(s)P\) has volume one, so Tonelli's theorem gives
\[
 (\lambda\times\nu)(\mathcal F)
 =\int_{\mathcal D}\lambda(\rho(s)P)\,d\nu(s)
 =\nu(\mathcal D)<\infty.
\]
The restriction of product Haar measure to \(\mathcal F\)
defines the invariant quotient measure by the standard
quotient-measure construction used in
\cref{prop:affinelattice}. Thus \(A_g\rtimes\Sp_{2g}(\Z)\) is a
lattice in the real affine group and inherits property~\(\rt\).
\end{proof}

\subsection{The first Torelli quotient}

Let
\[
 D^2\mathcal{T}_g=\{i\in\mathcal{T}_g:i^m\in[\mathcal{T}_g,\mathcal{T}_g]
                              \text{ for some }m\ge1\}.
\]
It is the inverse image of \(\tor(\mathcal{T}_g/[\mathcal{T}_g,\mathcal{T}_g])\) under \(\mathcal{T}_g\to\mathcal{T}_g/[\mathcal{T}_g,\mathcal{T}_g]\),
and hence is characteristic in \(\mathcal{T}_g\). Define
\[
 E_g=\overline{\mathcal{Q}}_{g,1}=\Mod(\Sigma_g)/D^2\mathcal{T}_g=\mathcal{Q}_{g,1}/\tor(\mathcal{T}_g/[\mathcal{T}_g,\mathcal{T}_g]).
\]
There is an exact sequence
\begin{equation}\label{eq:torsionfreeabelianextension}
 1\longrightarrow A_g\longrightarrow E_g
 \xrightarrow{p}\Sp_{2g}(\Z)\longrightarrow1.
\end{equation}

We now consider an affine realization of \(E_g\) constructed by Hain and Matsumoto.
The next statement is the closed-surface case of
\cite[Lemma~6.3]{HainMatsumoto}, and also the specialization of
\cref{prop:centralaffine} to \(z=-I\). We spell out the construction
in this case to identify the translation lattice and its index.

\begin{proposition}[The Hain--Matsumoto affine realization]
\label{prop:affinequotient}
There is an injective homomorphism
\[
 \phi:E_g\longrightarrow A_g\rtimes\Sp_{2g}(\Z)
\]
which induces the identity on the symplectic quotient and
multiplication by \(2\) on \(A_g\). Its image has index
\(2^{d_g}\). In particular, \(E_g\) is linear over \(\Z\)
and has property~\(\rt\).
\end{proposition}
\begin{proof}
Write \(E=E_g\) and \(A=A_g\).
Choose a lift \(u\in E\) of the central element \(z=-I\)
of \(\Sp_{2g}(\Z)\). The action of \(z\) on \(A\) is multiplication by
\(-1\), by \cref{lem:weights}. For \(x\in E\) let
\[
 b(x)=[x,u]=xux^{-1}u^{-1}.
\]
Since \(z\) is central in \(\Sp_{2g}(\Z)\), this commutator lies
in \(A\). We will write the law of \(A\) additively. The identity
\[
 [xy,u]=x[y,u]x^{-1}[x,u]
\]
and commutativity in \(A\) imply
\[
 b(xy)=b(x)+p(x)b(y).
\]
For \(a\in A\), the action of \(u\) gives
\(b(a)=a-za=2a\). Therefore
\(\phi(x)=(b(x),p(x))\) is a homomorphism to the stated
semidirect product. If it kills \(x\), then \(p(x)=1\),
so \(x=a\in A\), and \(2a=0\). Since \(A\) is torsion-free,
we obtain \(a=0\), which proves injectivity.

The subgroup \(\phi(E_g)\) surjects onto \(\Sp_{2g}(\Z)\) under the second-coordinate projection, and its intersection
with the translation subgroup is \(2A\). Every coset of the
image in \(A\rtimes\Sp_{2g}(\Z)\) therefore has a representative
in \(A\). Two such representatives determine the same coset
exactly when they differ by \(2A\). The index is
\([A:2A]=2^{d_g}\).

Choose a basis of \(A\), and
let \[r_g:\Sp_{2g}(\Z)\to\GL_{d_g}(\Z)\] 
be the action. The map
\[
 (a,S)\longmapsto
 \begin{pmatrix}
  r_g(S)&a&0\\
  0&1&0\\
  0&0&S
 \end{pmatrix}
 \quad\in\GL_{d_g+1+2g}(\Z)
\]
is an injective homomorphism. The lower block recovers \(S\),
and the translation column recovers \(a\), without requiring
faithfulness of \(r_g\). Finally, \cref{prop:arithmeticT}
and finite-index permanence give property~\(\rt\) of \(E\).
\end{proof}

\begin{remark}
An extension by the module \(A\) is described by a class
\(e\in H^2(\Sp_{2g}(\Z);A)\). Choose a set-theoretic
section \(s:\Sp_{2g}(\Z)\to E\) with \(s(1)=1\), and write
\[
 s(\gamma)s(\eta)=\alpha(\gamma,\eta)s(\gamma\eta),
 \qquad \alpha(\gamma,\eta)\in A.
\]
Associativity implies that \(\alpha\) is a \(2\)-cocycle and changing \(s\) changes it by a coboundary. Put
\(v(\gamma)=b(s(\gamma))\). Applying the crossed-homomorphism
identity and \(b|_A=2\) to the displayed product gives
\[
 2\alpha(\gamma,\eta)
 =v(\gamma)+\gamma v(\eta)-v(\gamma\eta).
\]
Thus \(2e=0\). Pushing out along \(A\xrightarrow{\times2}A\)
makes the extension split, yet it does not show that \(e=0\), and
does not give a section of \eqref{eq:torsionfreeabelianextension}.
The embedding of \cref{prop:affinequotient} corresponds to this
enlargement of the normal lattice.
\end{remark}

\begin{corollary}\label{cor:firststage}
The group \(\mathcal{Q}_{g,1}=\Mod(\Sigma_g)/[\mathcal{T}_g,\mathcal{T}_g]\) has property~\(\rt\)
for \(g\ge3\).
\end{corollary}
\begin{proof}
The kernel \(\tor(\mathcal{T}_g/[\mathcal{T}_g,\mathcal{T}_g])\) of \(\mathcal{Q}_{g,1}\to E_g\) is finite.
Since \(E_g\) has property~\(\rt\) by \cref{prop:affinequotient},
permanence under finite normal kernels gives the conclusion.
\end{proof}

\subsection{Higher Torelli quotients and proof of Theorem B}

We can now prove the property \(\rt\) for the Torelli lower central series quotients.

\begin{proof}[Proof of \cref{thm:main}]
By \cref{lem:nilab}, \(\mathcal{N}_{g,c}\) is nilpotent and normal
in \(\mathcal{Q}_{g,c}\), and
\[
 \mathcal{Q}_{g,c}/[\mathcal{N}_{g,c},\mathcal{N}_{g,c}]
 \cong \mathcal{Q}_{g,1}.
\]
The latter has property~\(\rt\) by \cref{cor:firststage}.
The nilpotent permanence theorem, \cref{thm:nilpotentpermanence},
therefore gives property~\(\rt\) of \(\mathcal{Q}_{g,c}\).
Finite generation is obvious.
By \cref{lem:nilab,lem:lattice}, the abelianization of \(\mathcal{N}_{g,c}\)
surjects onto the nonzero free abelian group \(A_g\), of rank
\(d_g>0\). Thus \(\mathcal{N}_{g,c}\), and hence \(\mathcal{Q}_{g,c}\), is infinite.
\end{proof}

Note that
\(\mathcal{N}_{g,c}\) itself does not have property~\(\rt\), as it
surjects onto \(\Z^{d_g}\). 

\begin{corollary}[Quotients with nilpotent Torelli image]
\label{cor:nilpotentimage}
Let \(f:\Mod(\Sigma_g)\to P\), with \(g\ge3\), be a
surjective homomorphism. If \(f(\mathcal{T}_g)\) is nilpotent of finite
class, then \(P\) has property~\(\rt\).
\end{corollary}
\begin{proof}
If the class is at most \(c\ge1\), then
\(f\bigl((\mathcal{T}_g)_{[c+1]}\bigr)
=\bigl(f(\mathcal{T}_g)\bigr)_{[c+1]}=1\).
Thus \(f\) factors through \(\mathcal{Q}_{g,c}\), and quotient
permanence applies. The trivial-image case factors through the
symplectic quotient and is covered as well.
\end{proof}

\begin{remark}
For \(g\ge3\) and \(c\ge1\), the torsion subgroup of
\(\mathcal{N}_{g,c}\) is finite and characteristic
\cite[Chapter~II]{Raghunathan}, hence normal in \(\mathcal{Q}_{g,c}\).
Thus \(\mathcal{Q}_{g,c}/\tor(\mathcal{N}_{g,c})\) also has
property~\(\rt\), with torsion-free nilpotent Torelli image.
This is the quotient obtained by replacing \((\mathcal{T}_g)_{[c+1]}\)
by its rational lower-central analogue: the elements of \(\mathcal{T}_g\)
having a positive power in \((\mathcal{T}_g)_{[c+1]}\).
\end{remark}

There are natural surjections
\[
 \cdots\longrightarrow \mathcal{Q}_{g,3}\longrightarrow \mathcal{Q}_{g,2}
 \longrightarrow \mathcal{Q}_{g,1}\longrightarrow\Sp_{2g}(\Z),
\]
where increasing \(c\) captures more of the Torelli group. The theorem
applies separately to each finite nilpotent stage.

\subsection{An obstruction for genus 2}
\label{sec:genustwo}

The hypothesis \(g\ge3\) in \cref{thm:main} cannot be removed.
In genus \(g=2\), every quotient in the corresponding family has a
finite-index subgroup with infinite abelianization. 
The obstruction comes from a theorem of McCarthy concerning the
first cohomology of level subgroups.

\begin{proposition}\label{prop:genustwo}
For every integer \(c\ge1\), the group \(\mathcal{Q}_{2,c}\) has a
finite-index subgroup admitting a surjective homomorphism onto
\(\Z\). In particular, \(\mathcal{Q}_{2,c}\) does not have
property~\(\rt\).
\end{proposition}
\begin{proof}
Let
\[
 \Gamma=\Mod(\Sigma_2)[2]
 =\ker\bigl(\Mod(\Sigma_2)\longrightarrow\Sp_4(\Z/2\Z)\bigr).
\]
 \(\Gamma\) has finite index in
\(\Mod(\Sigma_2)\) and contains \(\mathcal{T}_2\), since the Torelli group acts
trivially on integral first homology and hence on homology modulo
two. McCarthy's theorem \cite[Theorem~B; Theorem~2.1]{McCarthy}
gives
\[
 H^1(\Gamma;\Z)\ne0.
\]
Here the coefficients are trivial, so
\(H^1(\Gamma;\Z)=\Hom(\Gamma,\Z)\). Choose a nonzero
homomorphism \(\chi:\Gamma\to\Z\). Its image is an infinite
cyclic subgroup of \(\Z\). Identifying that image with \(\Z\),
we may assume that \(\chi\) is surjective.

Put \(K_c=(\mathcal{T}_2)_{[c+1]}\). Since \(c\ge1\), we have
\[
 K_c\subseteq[\mathcal{T}_2,\mathcal{T}_2]
     \subseteq[\Gamma,\Gamma]
     \subseteq\ker\chi.
\]
Thus \(\chi\) induces a surjection
\[
 \bar\chi:\Gamma/K_c\longrightarrow\Z.
\]
The subgroup \(K_c\) is normal in \(\Mod(\Sigma_2)\) and contained in
\(\Gamma\). Consequently \(\Gamma/K_c\) is a subgroup of
\(\mathcal{Q}_{2,c}\), and
\[
 [\mathcal{Q}_{2,c}:\Gamma/K_c]=[\Mod(\Sigma_2):\Gamma]<\infty.
\]
If \(\mathcal{Q}_{2,c}\) had property~\(\rt\), then its finite-index
subgroup \(\Gamma/K_c\), and hence its quotient \(\Z\), would
also have property~\(\rt\), which is impossible.
\end{proof}

The argument shows why passing to any of the nilpotent Torelli
quotients leaves the obstruction intact. Since
\(K_c\subseteq[\Gamma,\Gamma]\), the quotient map induces an
isomorphism
\[
 (\Gamma/K_c)_{\ab}\cong\Gamma_{\ab}.
\]
Thus the character supplied by McCarthy survives at every stage.

\begin{corollary}\label{cor:genustworelative}
For every \(c\ge1\), the pair \((\mathcal{Q}_{2,c},\mathcal{N}_{2,c})\) does not
have relative property~\(\rt\).
\end{corollary}
\begin{proof}
The subgroup \(\mathcal{N}_{2,c}\) is normal in \(\mathcal{Q}_{2,c}\), and
\[
 \mathcal{Q}_{2,c}/\mathcal{N}_{2,c}\cong\Sp_4(\Z).
\]
The latter group has property~\(\rt\). If the pair had relative
property~\(\rt\), then \cref{lem:normalextension} would give
property~\(\rt\) for \(\mathcal{Q}_{2,c}\), contrary to
\cref{prop:genustwo}.
\end{proof}

The symplectic quotient \(\Mod(\Sigma_2)/\mathcal{T}_2\cong\Sp_4(\Z)\) therefore
has property~\(\rt\), whereas none of the quotients
\(\mathcal{Q}_{2,c}\), \(c\ge1\), does. This shows the necessity
of the genus assumption for the property~\(\rt\) theorem.

\section{Automorphisms of free groups}
\label{sec:freeautomorphisms}

We now apply the general criterion to the homological representation
of a free-group automorphism group. This application is independent
of the mapping class group calculation: the integral module, the
contracting diagonal subgroup and the arithmetic ambient group are
different. The same argument nevertheless proves property~\(\rt\)
for every quotient in which the kernel of the homological action
has nilpotent image. We consider general rank \(n\ge3\) in the
preliminary calculations and specialize to \(n=3\) in the proof
of \cref{thm:autquotients}.

\subsection{The homological representation and its kernel}

Let \(F_n=\langle x_1,\ldots,x_n\rangle\) be the free group of
rank \(n\), and let
\[
 H_n=H_1(F_n;\Z)=F_n/(F_n)_{[2]}\cong\Z^n,
 \qquad H_n^*=\Hom(H_n,\Z).
\]
Write \(e_i=[x_i]\) for the standard basis of \(H_n\), and
\(e_i^*\) for its dual basis. Abelianization induces a surjection
\[
 q:\Aut(F_n)\longrightarrow\GL_n(\Z).
\]
Define
\[
 \operatorname{IA}_n=\ker q,
 \qquad
 \operatorname{SAut}(F_n)=q^{-1}(\SL_n(\Z)).
\]
Thus \(\operatorname{IA}_n\) consists of automorphisms acting
trivially on first homology, and \(\operatorname{SAut}(F_n)\)
has index two in \(\Aut(F_n)\). 

For \(c\ge1\), define
\[
 \mathcal{P}_{n,c}
   =\Aut(F_n)/(\operatorname{IA}_n)_{[c+1]}.
\]
The subgroup \((\operatorname{IA}_n)_{[c+1]}\) is characteristic
in \(\operatorname{IA}_n\), hence normal in \(\Aut(F_n)\).
Consequently there is an exact sequence
\begin{equation}\label{eq:autnilsequence}
 1\longrightarrow
 \operatorname{IA}_n/(\operatorname{IA}_n)_{[c+1]}
 \longrightarrow\mathcal{P}_{n,c}
 \longrightarrow\GL_n(\Z)\longrightarrow1.
\end{equation}
We prove \cref{thm:autquotients} below, beginning with the integral
abelianization of \(\operatorname{IA}_n\).

\subsection{The first Johnson homomorphism}

The commutator convention \([x,y]=xyx^{-1}y^{-1}\) gives the
standard identification
\[
 (F_n)_{[2]}/(F_n)_{[3]}
   \cong\bigwedge\nolimits^2H_n,
 \qquad [x,y]\longmapsto[x]\wedge[y],
\]
where \([x]\) and \([y]\) denote the images of \(x\) and \(y\) in the abelianization and the commutator on the left denotes its class modulo
\((F_n)_{[3]}\). This identification follows from the
commutator identities modulo third-order commutators since the
commutator is bilinear and alternating on abelianization, and
the classes \([x_j,x_k]\), \(j<k\), form a free abelian basis.
See \cite[Sections~3 and~4.1]{SatohSurvey}.

If \(\phi\in\operatorname{IA}_n\), then
\(\phi(x)x^{-1}\in(F_n)_{[2]}\) for every \(x\in F_n\).
Define
\begin{equation}\label{eq:autjohnsondefinition}
 \tau(\phi)([x])
   =[\phi(x)x^{-1}]
   \in(F_n)_{[2]}/(F_n)_{[3]}
   =\bigwedge\nolimits^2H_n.
\end{equation}
We check the conditions needed to regard this as an integral
homomorphism.

First, for \(x,y\in F_n\),
\[
 \phi(xy)(xy)^{-1}
 =\bigl(\phi(x)x^{-1}\bigr)
    x\bigl(\phi(y)y^{-1}\bigr)x^{-1}.
\]
Conjugation by \(x\) is trivial on
\((F_n)_{[2]}/(F_n)_{[3]}\). Thus the class of this expression
is the sum of the classes for \(x\) and \(y\). It follows that
\(x\mapsto[\phi(x)x^{-1}]\) is a homomorphism from \(F_n\)
to an abelian group, and therefore factors through \(H_n\).
This also proves that \eqref{eq:autjohnsondefinition} is
independent of the representative of \([x]\).

Second, every element of \(\operatorname{IA}_n\) acts trivially
on \((F_n)_{[2]}/(F_n)_{[3]}\), since under the exterior-square
identification its action is the exterior square of the identity
on \(H_n\). For \(\phi,\psi\in\operatorname{IA}_n\), the
identity
\[
 (\phi\psi)(x)x^{-1}
   =\phi\bigl(\psi(x)x^{-1}\bigr)
      \phi(x)x^{-1}
\]
therefore gives \[\tau(\phi\psi)=\tau(\phi)+\tau(\psi).\]
We thus obtain a homomorphism
\[
 \tau:\operatorname{IA}_n\longrightarrow W_n,
 \qquad
 W_n=\Hom\left(H_n,\bigwedge\nolimits^2H_n\right)
     =H_n^*\otimes\bigwedge\nolimits^2H_n.
\]
Finally, for \(\theta\in\Aut(F_n)\), direct substitution in
\eqref{eq:autjohnsondefinition} gives
\begin{equation}\label{eq:autjohnsonequivariance}
 \tau(\theta\phi\theta^{-1})
    =\left(\bigwedge\nolimits^2q(\theta)\right)\,
        \tau(\phi)\,q(\theta)^{-1}.
\end{equation}
Thus \(\tau\) is equivariant for the conjugation action and
the natural \(\GL_n(\Z)\)-action on \(W_n\).

\begin{proposition}[Integral abelianization of
\(\operatorname{IA}_n\)]\label{prop:iaabelianization}
For \(n\ge3\), the homomorphism \(\tau\) induces an isomorphism
of \(\GL_n(\Z)\)-modules
\[
 \operatorname{IA}_n/[\operatorname{IA}_n,\operatorname{IA}_n]
       \cong W_n.
\]
In particular, its kernel is
\([\operatorname{IA}_n,\operatorname{IA}_n]\), and the
abelianization has no torsion. Its rank is
\[
 d_n=n\binom n2=\frac{n^2(n-1)}2.
\]
\end{proposition}

This is the standard computation of the first Johnson
homomorphism and the abelianization of \(\operatorname{IA}_n\),
see \cite[Sections~4.2--4.3.1, especially equation~(4.3)]{SatohSurvey}
for the statement and its historical attributions. We recall how
the integral generators enter the computation. Magnus's theorem
\cite[Section~6]{Magnus1935} (see also
\cite[Theorem~4.1]{SatohSurvey}) states that
\(\operatorname{IA}_n\) is generated by
\[
 \begin{aligned}
 C_{ij}:&\quad x_i\longmapsto x_j^{-1}x_ix_j
       &&(i\ne j),\\
 M_{ijk}:&\quad x_i\longmapsto x_i[x_j,x_k]
       &&(i,j,k\text{ distinct},\ j<k),
 \end{aligned}
\]
with all other free generators fixed. In particular,
\(\operatorname{IA}_n\) is finitely generated. Their images
are
\[
 \tau(C_{ij})=e_i^*\otimes(e_i\wedge e_j),
 \qquad
 \tau(M_{ijk})=e_i^*\otimes(e_j\wedge e_k).
\]
Up to signs, these are exactly the basis vectors
\(e_i^*\otimes(e_j\wedge e_k)\), \(j<k\), of \(W_n\).
This explains surjectivity and the absence of a finite-index
ambiguity in the integral image. Equivariance is already given
by \eqref{eq:autjohnsonequivariance}, the assertion about the
kernel is the abelianization statement in
\cref{prop:iaabelianization}.

\subsection{The affine realization}

By \cref{prop:iaabelianization}, there exists a short exact sequence
\[
 1\longrightarrow W_n\longrightarrow\mathcal{P}_{n,1}
   \longrightarrow\GL_n(\Z)\longrightarrow1.
\]
The central element \(-I\in\GL_n(\Z)\) acts as \(-1\) on
\(H_n^*\) and as \(+1\) on \(\bigwedge^2H_n\). Consequently
it acts as \(-1\) on \(W_n\). This remains true when \(n\)
is odd. Indeed, in that case \(-I\) is outside \(\SL_n(\Z)\), but it
is still central in the full quotient \(\GL_n(\Z)\).

\begin{proposition}\label{prop:autaffinerealization}
There is an injective homomorphism
\[
 \mathcal{P}_{n,1}\longrightarrow W_n\rtimes\GL_n(\Z)
\]
whose image has index \(2^{d_n}\).
\end{proposition}
\begin{proof}
Apply \cref{prop:centralaffine} to the central element
\(-I\), for which \(1-(-I)=2I\) is injective on \(W_n\)
and has cokernel of order \(2^{d_n}\).

More explicitly, choose a lift \(t\in\mathcal{P}_{n,1}\) of
\(-I\), and define \(b(u)=[u,t]\). Centrality in the quotient
implies that \(b(u)\in W_n\), even though \(t\) need not be
central in \(\mathcal{P}_{n,1}\). The commutator identity
\[
 [uv,t]=u[v,t]u^{-1}[u,t]
\]
and the fact that \(W_n\) is abelian give
\[b(uv)=b(u)+q(u)b(v).\] 
Thus
\(u\mapsto(b(u),q(u))\) is a homomorphism to the semidirect
product. For \(w\in W_n\), written additively,
\[
 b(w)=w-twt^{-1}=2w.
\]
The kernel of the homomorphism is therefore trivial. Its image
projects onto \(\GL_n(\Z)\), and its intersection with the
translation subgroup is exactly \(2W_n\). Every coset of the
image has a representative in the translation subgroup, so the
index is \([W_n:2W_n]=2^{d_n}\).
\end{proof}

\subsection{Contracting weights and the arithmetic lattice}

Let
\[
 V_n=W_n\otimes\R,
 \qquad
 S_n=\{s\in\GL_n(\R):|\det s|=1\}.
\]
The representation \(\rho:S_n\to\GL(V_n)\) is given, in the
Hom description, by
\[
 \rho(s)f=\left(\bigwedge\nolimits^2s\right)f s^{-1}.
\]
The identity component of \(S_n\) is \(\SL_n(\R)\), of index
two. Hence \(S_n\) has property~\(\rt\) for \(n\ge3\).
It is also unimodular. Indeed, its modular homomorphism is
trivial on the semisimple group \(\SL_n(\R)\), and the
resulting homomorphism from the quotient of order two to
\(\R_{>0}\) must be trivial.

\begin{lemma}\label{lem:autweights}
For every \(n\ge3\), there is a diagonal one-parameter subgroup
\((a_t)_{t\in\R}\) of \(\SL_n(\R)\) such that every basis
vector \(e_i^*\otimes(e_j\wedge e_k)\), \(j<k\), of \(V_n\)
has nonzero weight. In particular, each such vector is
contracted either by \(\rho(a_t)\) as \(t\to+\infty\), or
by \(\rho(a_{-t})\) as \(t\to+\infty\).
\end{lemma}
\begin{proof}
Choose
\[
 \alpha_i=3^{i-1}\quad(1\le i<n),
 \qquad
 \alpha_n=-\sum_{i=1}^{n-1}3^{i-1},
 \qquad
 a_t=\operatorname{diag}
       (e^{\alpha_1t},\ldots,e^{\alpha_nt}).
\]
The exponents sum to zero, so \(\det a_t=1\). On the dual
basis vector \(e_i^*\), the multiplier is
\(e^{-\alpha_it}\), while on \(e_j\wedge e_k\) it is
\(e^{(\alpha_j+\alpha_k)t}\). Thus the weight in question is
\[
 \beta_{ijk}=-\alpha_i+\alpha_j+\alpha_k,
 \qquad j<k.
\]
We verify explicitly that it is nonzero.

If \(i=j\) or \(i=k\), cancellation leaves one of the
nonzero numbers \(\alpha_1,\ldots,\alpha_n\). Otherwise
\(i,j,k\) are distinct.

 If all three are less than \(n\),
equality \(\beta_{ijk}=0\) would express one power of three
as a sum of two distinct powers of three, which is impossible
by uniqueness of the base-three expansion. 

If \(i=n\), all
three summands \(-\alpha_n,\alpha_j,\alpha_k\) are positive.

Finally, if \(n\) is one of \(j,k\), let \(\ell\) denote the
other one. Then
\[
 \beta_{ijk}
 =-\alpha_i+\alpha_\ell
       -\sum_{r=1}^{n-1}\alpha_r<0,
\]
because \(i\ne\ell\) and all the numbers in this last sum
are positive. These cases exhaust the possibilities.

On a vector of weight \(\beta\), the action is multiplication
by \(e^{\beta t}\). If \(\beta<0\), take \(t\to+\infty\).
If \(\beta>0\), take \(t\to-\infty\). In both cases the
multiplier tends to zero.
\end{proof}

For example, when \(n=3\), the chosen exponents are
\((1,3,-4)\). If \(i,j,k\) are distinct, the trace-zero
condition gives \(\beta_{ijk}=-2\alpha_i\), which is
nonzero. If an index cancels, the weight is one of
\(1,3,-4\). Thus the rank-three case needed for
\cref{thm:autquotients} is covered by the same argument.

\begin{proposition}\label{prop:autaffineT}
For \(n\ge3\), the groups \[V_n\rtimes S_n \text{\ \ \ \  and\ \ \ \ \ }
W_n\rtimes\GL_n(\Z)\] have property~\(\rt\).
\end{proposition}
\begin{proof}
The vectors contracted in \cref{lem:autweights} form a basis
of \(V_n\). Applying \cref{prop:affineT} with the subgroup
\(S_n\) therefore proves property~\(\rt\) for
\(V_n\rtimes S_n\). In terms of unitary representations,
property~\(\rt\) of \(S_n\) first supplies an
\(S_n\)-invariant vector. The contraction argument forces
that vector to be invariant under each translation along a
basis vector, hence under all of \(V_n\).

To pass to the integral group, we check volume preservation.
Write \(r_n=\binom n2\), so \(\dim V_n=nr_n\). The
determinant identities for an exterior square and a tensor
product give
\[
 \begin{aligned}
 \det(\rho(s))
 &=\det(s^{-T})^{r_n}
       \det\left(\bigwedge\nolimits^2s\right)^{n}\\
 &=(\det s)^{-r_n}(\det s)^{n(n-1)}
   =(\det s)^{r_n}.
 \end{aligned}
\]
For completeness, in the exterior-square determinant each
eigenvalue of \(s\) occurs in exactly \(n-1\) pairs, so
\[\det\left(\bigwedge^2s\right)=(\det s)^{n-1}.\]
For operators
\(u\) and \(v\) on spaces of dimensions \(p\) and \(q\),
respectively, the eigenvalues of \(u\otimes v\) are their
pairwise products, yielding
\[\det(u\otimes v)=(\det u)^q(\det v)^p.\]
These identities can be established by triangularizing over
\(\C\), without assuming diagonalizability.

Since \(s\in S_n\), we conclude
\[|\det\rho(s)|=1.\] Thus \(\rho(s)\) preserves Lebesgue
measure on \(V_n\), including when it reverses orientation.
The group \(\GL_n(\Z)\) preserves the lattice \(W_n\), and
is a lattice in \(S_n\). The latter assertion follows from
the classical lattice \(\SL_n(\Z)\subseteq\SL_n(\R)\) by passing
to these finite extensions, see \cite[Appendix~B]{BHV}.

All hypotheses of \cref{prop:affinelattice} thus hold. More
explicitly, normalize Lebesgue measure so that a fundamental
parallelepiped for \(W_n\) has volume one, and multiply it
by Haar measure on \(S_n\). This is Haar measure on the
affine group. If \(\mathcal D\) is a finite-volume right
fundamental domain for \(\GL_n(\Z)\) in \(S_n\), the affine
fundamental domain has fiber \(\rho(s)P\) over \(s\in
\mathcal D\), where \(P\) is the parallelepiped. Each fiber
has volume one, so the total volume is the Haar measure of
\(\mathcal D\), which is finite. Thus
\(W_n\rtimes\GL_n(\Z)\) is a lattice in the real affine
group and inherits property~\(\rt\).
\end{proof}

\subsection{Passing to higher nilpotent stages}

\begin{proof}[Proof of \cref{thm:autquotients}]
By \cref{prop:autaffinerealization,prop:autaffineT}, the group
\(\mathcal{P}_{3,1}\) is isomorphic to a finite-index
subgroup of a group with property~\(\rt\), and hence has
property~\(\rt\).

Fix \(c\ge1\) and let \[J=\operatorname{IA}_3/
(\operatorname{IA}_3)_{[c+1]}.\] Magnus's finite-generation
theorem \cite[Section~6]{Magnus1935} shows that \(J\) is finitely generated. Its
\((c+1)\)-st lower-central subgroup is trivial, and therefore
it is nilpotent of class at most \(c\). Furthermore,
\[
 [J,J]=(\operatorname{IA}_3)_{[2]}/
                (\operatorname{IA}_3)_{[c+1]},
 \qquad
 \mathcal{P}_{3,c}/[J,J]\cong\mathcal{P}_{3,1}.
\]
The nilpotent quotient theorem \cref{thm:quotienttower} gives
property~\(\rt\) for \(\mathcal{P}_{3,c}\). In other words, this verifies
all hypotheses of \cref{thm:generalcriterion}  with \(G=\Aut(F_3)\), \(K=\operatorname{IA}_3\),
\(A=W_3\), \(\Gamma=\GL_3(\Z)\), \(S=S_3\), and the
central element \(-I\).

The same commutator identity gives
\(J/[J,J]\cong W_3\), which is an infinite free abelian
group. It follows that \(J\) is infinite. The group
\(\mathcal{P}_{3,c}\) is finitely generated because
\(\Aut(F_3)\) is generated by finitely many Nielsen
automorphisms, and is infinite also because it surjects onto
\(\GL_3(\Z)\).

Finally, suppose that \(f:\Aut(F_3)\to E\) is a surjective
homomorphism for which \(f(\operatorname{IA}_3)\) is nilpotent
of class at most \(c\). Then
\(f((\operatorname{IA}_3)_{[c+1]})=1\), so \(f\) factors
through \(\mathcal{P}_{3,c}\). Quotient permanence proves
property~\(\rt\) for \(E\).
\end{proof}

\subsection{The Andreadakis filtration and the examples of
Lubotzky--Pak}

There is a related, but generally different, family of quotients
obtained by allowing automorphisms to act on nilpotent quotients
of the free group itself. For \(r\ge1\), define
\[
 \mathcal A_n(r)=\ker\!\left(
 \Aut(F_n)\longrightarrow
 \Aut\bigl(F_n/(F_n)_{[r+1]}\bigr)\right).
\]
This is the Andreadakis filtration \cite{Andreadakis}. Its first term is
\(\operatorname{IA}_n\), and
\cref{prop:iaabelianization} gives
\(\mathcal A_n(2)=[\operatorname{IA}_n,\operatorname{IA}_n]\).
Its centrality means that
\[
 [\mathcal A_n(r),\mathcal A_n(s)]
      \subseteq\mathcal A_n(r+s),
 \qquad
 (\operatorname{IA}_n)_{[r]}\subseteq\mathcal A_n(r).
\]
The first inclusion is the standard commutator property of the
Andreadakis filtration, see \cite[Section~3.2]{SatohSurvey}.
The second follows by induction from the first, starting with
\(\mathcal A_n(1)=\operatorname{IA}_n\).

\begin{corollary}\label{cor:tamefreequotients}
For \(n\ge3\) and \(r\ge1\), the image
\[
 \im\!\left(\Aut(F_n)\longrightarrow
        \Aut\bigl(F_n/(F_n)_{[r+1]}\bigr)\right)
\]
has property~\(\rt\). The corresponding image of
\(\operatorname{SAut}(F_n)\) has property~\(\rt\) as well.
\end{corollary}
\begin{proof}
For \(r=1\), the two images are \(\GL_n(\Z)\) and
\(\SL_n(\Z)\). For \(r\ge2\), the inclusion
\((\operatorname{IA}_n)_{[r]}\subseteq\mathcal A_n(r)\)
gives a surjection
\[
 \mathcal{P}_{n,r-1}\longrightarrow
        \Aut(F_n)/\mathcal A_n(r).
\]
The group on the right is the image in the statement. For \(n=3\),
it has property~\(\rt\) by \cref{thm:autquotients}, while for \(n\ge4\),
this follows from property~\(\rt\) of \(\Aut(F_n)\) and quotient
permanence \cite{KalubaNowakOzawa,KalubaKielakNowak,NitscheAutFour}.
The special image is its subgroup of index two, detected by determinant
on first homology, and consequently also has property~\(\rt\).
\end{proof}

The special-image assertion of \cref{cor:tamefreequotients}
was proved by Lubotzky and Pak
\cite[Theorem~3.8 and Section~4]{LubotzkyPak}. Their proof
realizes that group as a lattice in a semidirect product of
a real unipotent group and \(\SL_n(\R)\), and proves
property~\(\rt\) for this ambient group.  
Conversely, their theorem at \(r=2\), together
with \(\mathcal A_n(2)=[\operatorname{IA}_n,
\operatorname{IA}_n]\) and \cref{thm:quotienttower}, also
proves \cref{thm:autquotients}. The explicit first-stage
argument above identifies all of the ingredients without
using the higher-stage theorem of \cite{LubotzkyPak}.

The image in \cref{cor:tamefreequotients} is called the
\emph{tame automorphism group} of the free nilpotent group.
 In general, an automorphism of
\(F_n/(F_n)_{[r+1]}\) need not lift to an automorphism of
\(F_n\). Lubotzky and Pak explicitly distinguish the image from
the full automorphism group and note that the image has
infinite index for \(n\ge2\) and \(r\ge4\)
\cite[Example~1.14]{LubotzkyPak}.

\section{More examples}
\label{sec:finalremarks}

We now briefly describe further applications of \cref{thm:strategy}.
In each case the main point is to identify an abelianized kernel
to which the affine argument applies, or to reduce to a quotient
already considered above.

\subsection{Free products of free abelian groups}

Let
\[
 A=\Z^{n_1}*\cdots *\Z^{n_r},\qquad r\ge3,\quad n_i\ge3,
\]
and let \(G=\operatorname{Out}(A)\). Denote by \(K\) the kernel
of the action of \(G\) on \(H_1(A;\Z)\). Then
\[
 G/K_{[c+1]}\text{ has property }\rt
 \qquad\text{for every }c\ge1.
\]
To see this, pass to the finite-index subgroup \(G_0\) preserving
the conjugacy class of each free factor. Its homological quotient
is \(\prod_i\GL_{n_i}(\Z)\), and automorphisms of the individual
factors give a section. The kernel \(K\) is generated by partial
conjugations: one conjugates a free factor by an element of another
and fixes the remaining factors. The standard generating-set
description is recalled in \cite[Section~2.5]{SaleRAAG}.
It follows that \(K/[K,K]\) is an equivariant quotient of
\[
 \bigoplus_{i=1}^r(\Z^{n_i})^{\oplus(r-1)},
\]
with the \(i\)-th arithmetic factor acting through its standard
representation on the corresponding summands.

The contraction argument and lattice passage in
\cref{sec:generalaffine} give property~\(\rt\) for the semidirect
product of this module with \(\prod_i\GL_{n_i}(\Z)\).
Its quotient \(G_0/[K,K]\) therefore has property~\(\rt\).
 \cref{thm:quotienttower}  together with passing via a finite index subgroup give the
assertion. The same reasoning applies to \(\Aut(A)\).

These are further examples in which the original group does not
have property~\(\rt\). Indeed, the defining graph is a disjoint
union of cliques and satisfies the largeness criterion of
\cite[Theorem~2(1b)]{GuirardelSale}. Thus \(G\) has a finite-index
subgroup mapping onto a nonabelian free group.

\subsection{Johnson-filtration quotients}
\label{sec:johnsonfiltrationquotients}

Let \(g\ge3\) and \(\pi=\pi_1(\Sigma_g)\). For \(k\ge1\),
the action on the surface group induces
\[
 \rho_k:\Mod(\Sigma_g)\longrightarrow
 \operatorname{Out}\bigl(\pi/\pi_{[k+1]}\bigr).
\]
The kernels form the Johnson filtration. Its commutator property
gives \((\mathcal T_g)_{[k]}\subseteq\ker\rho_k\),
see \cite{HabiroMassuyeauJohnson}.
For \(k\ge2\), the image of \(\rho_k\) is therefore a quotient
of \(\mathcal Q_{g,k-1}\), and has property~\(\rt\) by
\cref{thm:main}. For \(k=1\), the image is \(\Sp_{2g}(\Z)\),
which also has property~\(\rt\).

These classical quotients record the action on nilpotent
approximations to the surface group, rather than to the Torelli
group. They are linear and residually finite: since
\(\pi/\pi_{[k+1]}\) is finitely generated and nilpotent, its
outer automorphism group is arithmetic by
\cite[Theorem~1.1]{BauesGrunewald}. Linearity and Malcev's theorem
then apply to the finitely generated subgroup
\(\rho_k(\Mod(\Sigma_g))\).

\subsection{Quotients associated with the Johnson kernel}

Let
\[
 \mathcal J_g=\ker\!\left(\mathcal T_g
       \xrightarrow{\tau_J}\Lambda^3H_g/j(H_g)\right).
\]
For \(g\ge4\), our method also gives property~\(\rt\) for
\[
 \Mod(\Sigma_g)/(\mathcal J_g)_{[c+1]},\qquad c\ge1.
\]
The additional ingredients are finite generation of \(\mathcal J_g\),
proved in \cite[Theorem~A]{ChurchErshovPutman}, and unipotence of
the action of \(\mathcal T_g/\mathcal J_g\) on
\(H_1(\mathcal J_g;\Q)\), proved in
\cite[Theorem~A]{DimcaHainPapadima}.

For the first stage, put
\[
 E=\Mod(\Sigma_g)/[\mathcal J_g,\mathcal J_g],
 \qquad N=\mathcal T_g/[\mathcal J_g,\mathcal J_g].
\]
The abelian normal subgroup \(\mathcal J_g/[\mathcal J_g,\mathcal J_g]\)
has finite torsion, and the quotient of \(N\) by this subgroup
is the free abelian group \(\mathcal T_g/\mathcal J_g\).
Unipotence implies that some power of the augmentation ideal of
\(\Q[\mathcal T_g/\mathcal J_g]\) annihilates
\(H_1(\mathcal J_g;\Q)\). After removing torsion,
the same identity holds on the integral lattice. Since the
operators \(t-1\) describe commutators with lifts of \(t\),
it follows that \(N\) is finite-by-nilpotent. Moreover,
\[
 E/[N,N]\cong\Mod(\Sigma_g)/[\mathcal T_g,\mathcal T_g].
\]
The first-stage result for the Torelli group and
\cref{rem:nilpotentextensions} give property~\(\rt\) for \(E\).
The result then follows from \cref{thm:quotienttower} with \(K=\mathcal J_g\).

The interest here is the use of a different natural normal
subgroup, rather than an unrelated source of property~\(\rt\):
after removal of a finite kernel, each quotient is a quotient
of a sufficiently deep Torelli lower-central stage.
In genus three, \cite[Theorems~A and~E]{GaifullinJohnson} gives
finite generation of \((\mathcal J_3)_{\ab}\) and unipotence
after restriction to a finite-index subgroup of
\(\mathcal T_3/\mathcal J_3\). Extending the full-action argument
to this case requires a further step, the conclusion above
is stated only for \(g\ge4\).

\subsection{Marked points and boundary components}

Let \(\Sigma_{g,r}^{p}\) have \(p\) marked points and \(r\)
boundary components, all fixed pointwise by mapping classes.
Define \(\mathcal T_{g,r}^{p}\) as the kernel of the action on
the first homology of the capped closed surface. This convention
is important when there is more than one boundary component.
For \(g\ge3\), the same argument gives property~\(\rt\) for
\[
 \Mod(\Sigma_{g,r}^{p})/
             (\mathcal T_{g,r}^{p})_{[c+1]},\qquad c\ge1.
\]
Indeed, writing \(H_{g,\Q}=H_g\otimes_\Z\Q\), one has
\[
 H_1(\mathcal T_{g,r}^{p};\Q)
 \cong
 \bigl(\Lambda^3H_{g,\Q}/j(H_{g,\Q})\bigr)
       \oplus H_{g,\Q}^{\oplus(p+r)}.
\]
See \cite[Propositions~3.5 and~4.6]{Hain}.
Finite generation, finite torsion in the integral abelianization,
and algebraicity of its torsion-free quotient are recorded in
\cite[proof of Proposition~3.4]{HainMatsumoto}.
Every summand has the required contraction property, and
\(-I\) acts as \(-1\), and consequently  \cref{thm:strategy} applies.

There is an additional geometric feature. For \(r>0\) and
\(c\ge2\), the boundary twists generate a central subgroup
isomorphic to \(\Z^r\) in the displayed quotient.
With one boundary, its twist acts on \(F_{2g}\) by conjugation
by \(\omega=\prod_i[a_i,b_i]\). It has infinite order on
\(F_{2g}/(F_{2g})_{[4]}\): for a nonzero power, the displacement
of \(a_1\) has nonzero degree-three class
\(m[\sum_i[a_i,b_i],a_1]\) in the free Lie algebra.
On the other hand, the third lower-central subgroup of the
Torelli group acts trivially on this nilpotent quotient, by
the commutator property of the Johnson filtration.
Capping all but one boundary and forgetting the marked points
detects the boundary twists independently. Thus the construction
also gives property-\(\rt\) groups with central
free abelian subgroups of arbitrarily large rank.

\subsection{Outer automorphisms of the free group of rank three}

Set
\[
 \operatorname{IO}_3=
 \ker\!\left(\operatorname{Out}(F_3)\longrightarrow\GL_3(\Z)\right).
\]
The natural map \(\Aut(F_3)\to\operatorname{Out}(F_3)\) maps
\(\operatorname{IA}_3\) onto \(\operatorname{IO}_3\), and maps
their lower-central terms onto one another. Hence it induces
a surjection
\[
 \mathcal P_{3,c}\longrightarrow
 \operatorname{Out}(F_3)/(\operatorname{IO}_3)_{[c+1]}.
\]
These quotients have property~\(\rt\) by
\cref{thm:autquotients}. They are finitely generated as quotients of
\(\Aut(F_3)\), and infinite because they map onto \(\GL_3(\Z)\).
This places the free-group application in the outer-automorphism
setting, which is closer to the closed-surface case. 

\subsection{Connected sums of \texorpdfstring{\(S^1\times S^2\)}{S1 x S2}}
\label{sec:three-dimensional-connected-sums}

Let \(Y_g=\#^g(S^1\times S^2)\), and let \(\Mod(Y_g)\)
denote its orientation-preserving smooth mapping class group.
Since \(\pi_1(Y_g)\cong F_g\), the action on the fundamental
group gives a homomorphism to \(\operatorname{Out}(F_g)\).
Laudenbach's theorem \cite{Laudenbach} gives the exact sequence
\[
 1\longrightarrow(\Z/2\Z)^g\longrightarrow\Mod(Y_g)
 \xrightarrow{\rho}\operatorname{Out}(F_g)\longrightarrow1,
\]
whose finite kernel is generated by sphere twists. See also
\cite[Theorem~2.1 and Corollary~5.2]{BrendleBroaddusPutman}
for a modern account. Define the homological kernel
\[
 K_g=\ker\bigl(\Mod(Y_g)\longrightarrow\GL_g(\Z)\bigr),
\]
where the map is the action on
\(H_1(Y_g;\Z)\cong\Z^g\). Then
\[
 \Mod(Y_g)/(K_g)_{[c+1]}\text{ has property }\rt
 \qquad(g\ge3,\ c\ge1).
\]

For \(g=3\), the map \(\rho\) sends \(K_3\) onto
\(\operatorname{IO}_3\), and sends their lower-central
terms onto one another. It therefore induces a surjection
\[
 \Mod(Y_3)/(K_3)_{[c+1]}\longrightarrow
 \operatorname{Out}(F_3)/(\operatorname{IO}_3)_{[c+1]}
\]
with finite kernel, namely the image of the sphere-twist
subgroup. The target has property~\(\rt\) by the preceding
subsection and \cref{thm:autquotients}. Finite-kernel
permanence proves the assertion.

For \(g\ge4\), the whole group \(\Mod(Y_g)\) already has
property~\(\rt\): its quotient \(\operatorname{Out}(F_g)\)
has property~\(\rt\), as a quotient of \(\Aut(F_g)\)
\cite{KalubaKielakNowak,NitscheAutFour}, and the kernel of
\(\rho\) is finite. In rank three, by contrast,
\(\Mod(Y_3)\) does not have property~\(\rt\), since it
surjects onto \(\operatorname{Out}(F_3)\), which is not Kazhdan
\cite{McCool,GrunewaldLubotzky}. Thus the rank-three
application gives a three-dimensional analogue of the
nilpotent-kernel quotients considered above, rather than
property~\(\rt\) for the original mapping class group.

For \(g\ge4\), the full group \(\Mod(Y_g)\) is nonlinear
over every field: the splitting of the Laudenbach sequence
\cite[Theorem~A]{BrendleBroaddusPutman} embeds
\(\operatorname{Out}(F_g)\), which contains the nonlinear
group \(\Aut(F_3)\) \cite{FormanekProcesi}.

\subsection{The kernel of the Casson core}
\label{sec:casson}

The Casson invariant \(\lambda(M)\in\Z\) is an invariant of
closed oriented integral homology \(3\)-spheres. Fix a Heegaard
splitting of \(S^3\) into two handlebodies with common boundary
\(\Sigma_g\). Regluing by \(\varphi\in\mathcal T_g\) produces
an integral homology sphere \(M_\varphi\). Recall that
\(\mathcal J_g=\ker\tau_J\) is the Johnson kernel. Morita proved
that \(\varphi\mapsto\lambda(M_\varphi)\), although generally
not a homomorphism on \(\mathcal T_g\), restricts to a
homomorphism on \(\mathcal J_g\) \cite{MoritaCasson}.

Morita also constructed a nonzero homomorphism
\(d_1:\mathcal J_g\longrightarrow\Z\), the \emph{core of the
Casson invariant}, which is invariant under conjugation by
\(\Mod(\Sigma_g)\) \cite{MoritaCassonCore}. This core is independent
of the chosen splitting; the gluing-dependent Casson homomorphism
is recovered from it, up to normalization, by adding a term
obtained from the second Johnson homomorphism
\cite[Theorem~7.1]{MoritaSakasaiSuzuki}.
Consequently \(\ker d_1\) is normal in \(\Mod(\Sigma_g)\), and
there is a central exact sequence
\[
 1\longrightarrow d_1(\mathcal J_g)
 \longrightarrow\Mod(\Sigma_g)/\ker d_1
 \longrightarrow\Mod(\Sigma_g)/\mathcal J_g
 \longrightarrow1.
\]
Its kernel \(d_1(\mathcal J_g)\) is infinite cyclic.

\begin{proposition}\label{prop:casson}
For \(g\ge3\), the group \(\Mod(\Sigma_g)/\ker d_1\)
has property~\(\rt\).
\end{proposition}
\begin{proof}
Since the Johnson homomorphism has abelian target,
\([\mathcal T_g,\mathcal T_g]\subseteq\mathcal J_g\).
For \(x\in\mathcal T_g\) and
\(y\in[\mathcal T_g,\mathcal T_g]\), conjugation invariance gives
\(d_1(xyx^{-1}y^{-1})=d_1(y)-d_1(y)=0\).
Thus \((\mathcal T_g)_{[3]}\subseteq\ker d_1\), so the group
in the statement is a quotient of \(\mathcal Q_{g,2}\).
The conclusion follows from \cref{thm:main} and permanence of
property~\(\rt\) under quotients.
\end{proof}

\subsection{Quotients defined by clasper surgery}
\label{sec:clasperquotients}

For \(k\ge2\), let \(R_{g,k}\) consist of the mapping classes
whose mapping cylinders are \(Y_k\)-equivalent to the trivial
cylinder. The mapping cylinder of \(f\) is
\(\Sigma_g\times[0,1]\) with one boundary parametrization
changed by \(f\);\(Y_k\)-equivalence is generated by
degree-\(k\) clasper surgeries. Compatibility with stacking
and boundary reparametrization makes \(R_{g,k}\) normal in
\(\Mod(\Sigma_g)\).
By \cite[Corollary~5.9]{HabiroMassuyeauSurvey},
\((\mathcal T_g)_{[k]}\subseteq R_{g,k}\). Hence
\[
 \Mod(\Sigma_g)/R_{g,k}\text{ has property }\rt
 \qquad(g\ge3,\ k\ge2),
\]
as a quotient of \(\mathcal Q_{g,k-1}\).

This gives a three-dimensional interpretation of further Kazhdan
quotients, through clasper surgery and finite-type invariants.
In the one-boundary version, the degree-two theory involves the
nilpotent action, an Alexander-type invariant and the Casson
invariant; Morita's Casson core also extends to homology cylinders
\cite[Theorems~A and~D]{MassuyeauMeilhan}.

\subsection{Metabelian Torelli images}
\label{sec:metabeliantorelli}

Write
\[
 \mathcal T_g''=
 [[\mathcal T_g,\mathcal T_g],[\mathcal T_g,\mathcal T_g]].
\]
Then
\[
 \Mod(\Sigma_g)/\mathcal T_g''\text{ has property }\rt
 \qquad(g\ge12).
\]
Indeed, Ershov and He prove that the maximal metabelian quotient
\(\mathcal T_g/\mathcal T_g''\) is nilpotent in this range
\cite[Corollary~1.5]{ErshovHe}, so \cref{cor:nilpotentimage}
applies. Their statement for a surface with one boundary component
passes to the closed surface by capping, which maps the Torelli
groups and their commutator subgroups surjectively.
Thus requiring the Torelli image to be metabelian already forces
the finite-class nilpotence needed by our argument.

More generally, the same argument shows that
\(\mathcal T_g/[(\mathcal T_g)_{[k]},(\mathcal T_g)_{[k]}]\)
is nilpotent when \(k\ge2\) and \(g\ge8k-4\). Consequently
\[
 \Mod(\Sigma_g)/
 [(\mathcal T_g)_{[k]},(\mathcal T_g)_{[k]}]
 \text{ has property }\rt
 \qquad(g\ge8k-4).
\]
These quotients retain the abelianization of a deeper
lower-central subgroup instead of killing that subgroup.

\subsection{Mapping class groups of sphere-product connected sums}
\label{sec:sphereproducts}

Let \(W_g^{2n}=\#^g(S^n\times S^n)\), and write
\(\Mod(W_g^{2n})=\pi_0\operatorname{Diff}^{+}(W_g^{2n})\)
for its orientation-preserving smooth mapping class group.

\begin{proposition}\label{thm:sphereproducts}
The group \(\Mod(W_g^{2n})\) has property~\(\rt\) if
\(n\ge3\) is odd and \(g\ge2\), or if \(n\ge4\) is even
and \(g\ge3\). In particular, the conclusion holds for all
\(n\ge3\) and \(g\ge3\).
\end{proposition}

\begin{proof}
Let \(H_g=H_n(W_g^{2n};\Z)\cong\Z^{2g}\), and denote the
kernel and image of the homological action by \(K_g\) and
\(\Gamma_g\), respectively. Kreck's description
\cite[Proposition~3]{Kreck}, recalled in
\cite[Sections~1.1--1.3]{Krannich}, gives
\[
 \Gamma_g=
 \begin{cases}
  \operatorname{O}_{g,g}(\Z),&n\text{ even},\\
  \Sp_{2g}(\Z),&n=3,7,\\
  \Sp_{2g}^{q}(\Z),&n\text{ odd},\ n\ne3,7.
 \end{cases}
\]
Here the orthogonal group preserves the hyperbolic form
\(\left(\begin{smallmatrix}0&I_g\\I_g&0\end{smallmatrix}\right)\);
the theta subgroup is the stabilizer of
\(q(x,y)=\sum_i x_i y_i\pmod2\), and has finite index in
\(\Sp_{2g}(\Z)\). The arithmetic groups are lattices in
\(\operatorname{O}_{g,g}(\R)\) for even \(n\) and
\(\Sp_{2g}(\R)\) for odd \(n\). These real groups, and hence
their lattices, have property~\(\rt\) in the stated ranges
\cite{BorelHarishChandra,BHV}.

There is also an exact sequence
\[
 1\longrightarrow F_g\longrightarrow K_g
 \longrightarrow H_g\otimes_\Z C_n\longrightarrow1,
 \qquad
 C_n=\im\bigl(\pi_n\operatorname{SO}(n)
       \longrightarrow\pi_n\operatorname{SO}(n+1)\bigr),
\]
where \(F_g\) is finite and central in \(\Mod(W_g^{2n})\).
The conjugation action is standard on \(H_g\) and trivial
on \(C_n\). Moreover, \(C_n\cong\Z\) if \(n\equiv3\pmod4\),
and is finite otherwise
\cite[Lemmas~1.2--1.3 and Table~1]{Krannich}.
These descriptions imply finite generation of \(K_g\) and
\(\Mod(W_g^{2n})\), and \([K_g,K_g]\subseteq F_g\) gives
\((K_g)_{[3]}=1\).
If \(C_n\) is finite then so is \(K_g\) and passing through the finite kernel proves the claim.

Otherwise, the finite central kernel gives
\[
 (K_g/[K_g,K_g])\otimes_\Z\R\cong H_g\otimes_\Z\R
 \cong\R^{2g}
\]
with its standard symplectic action. This action extends to
\(\Sp_{2g}(\R)\) and has determinant one. In a symplectic basis,
\[
 a_t=\operatorname{diag}(e^{-t}I_g,e^tI_g),\qquad
 a_t e_i\longrightarrow0,\quad a_{-t}f_i\longrightarrow0
 \quad(t\to+\infty).
\]
Thus the contraction hypothesis holds. The central element
\(-I\) lies in both possible symplectic arithmetic groups and
acts by \(-1\). Theorem~\ref{thm:strategy}, with \(c=2\),
therefore gives property~\(\rt\) for \(\Mod(W_g^{2n})\).
\end{proof}

For odd \(n\ge3\) and \(g\ge2\), the same conclusion holds
for \(W_g^{2n}\#\Sigma\), where \(\Sigma\) is a smooth homotopy
\(2n\)-sphere: its mapping class group is a quotient of
\(\Mod(W_g^{2n})\) by a central subgroup of order at most two
\cite[Lemma~3 and Theorem~3]{Kreck}
\cite[Introduction]{Krannich}.
Both mapping class groups are not residually finite and hence
nonlinear over every field. Indeed, the intersection
of the kernels of all finite quotients of \(\Mod(W_g^{2n})\)
has order greater than two
\cite[Theorem~2.3 and Remark~2.4(i)]{KrannichKupersMezher}
\cite{KervaireMilnor}, so its image remains nontrivial after
taking the exotic-sphere quotient.

\subsection{Complex threefolds and related six-manifolds}
\label{sec:threefolds}

Let \(M\) be a closed, simply connected, oriented smooth
\(6\)-manifold with \(H_2(M;\Z)\cong\Z\) and \(x^3\ne0\)
for a generator \(x\in H^2(M;\Z)\). Then
\[
 b_3(M)\ge4\quad\Longrightarrow\quad
 \Mod(M)\text{ has property }\rt,
\]
where \(\Mod(M)\) denotes the orientation-preserving smooth
mapping class group. To see this, write \(b_3(M)=2g\) and
let \(K\) be the kernel of its action on \(H^3(M;\Z)\).
Kreck--Su \cite[Theorems~2.3 and~2.4]{KreckSuThreefolds}
give an exact sequence
\[
 1\longrightarrow K\longrightarrow\Mod(M)
 \longrightarrow\Sp_{2g}(\Z)\longrightarrow1.
\]
They identify \(K\) as a finite central extension of a
finitely generated abelian group, with
\[
 H_1(K;\R)\cong H^3(M;\R)
\]
equivariantly for the standard symplectic action. In
particular, \(K_{[3]}=1\). The contraction calculation in
the proof of \cref{thm:sphereproducts} applies to this module,
and the central element \(-I\) acts as minus the identity.
Since \(g\ge2\), all hypotheses of \cref{thm:strategy} hold.
Taking \(c=2\) proves property~\(\rt\) for \(\Mod(M)\)
itself.

The hypotheses include smooth complete-intersection
threefolds with \(b_3\ge4\), in particular cubic and quintic
hypersurfaces in \(\mathbb{CP}^4\), whose third Betti numbers
are \(10\) and \(204\), respectively. They also include
\(\mathbb{CP}^3\#^g(S^3\times S^3)\) for \(g\ge2\).
For the cubic and quintic, Randal-Williams's computation
\cite[Theorem~B]{RandalWilliamsHypersurfaces} shows that
the smooth mapping class group is not residually finite.
These finitely generated groups are therefore nonlinear
over every field \cite{Nica}. Here one again obtains
property~\(\rt\) for the entire mapping class group.

\subsection{Simply connected spin five-manifolds}
\label{sec:spin-five-manifolds}

Let \(M\) be a closed, simply connected spin \(5\)-manifold.
Suppose that \(H_2(M;\Z)\) has no \(2\)- or \(3\)-torsion
and has free rank \(g\ge3\). Then \(\Mod(M)\) has
property~\(\rt\). The structural input is Jin's description
\cite[Theorem~1.1]{JinSpinFive} of the exact sequence
\[
 1\longrightarrow K\longrightarrow\Mod(M)
 \longrightarrow\Aut(H_2(M;\Z),L_M)\longrightarrow1,
\]
where \(L_M\) is the torsion linking pairing and \(K\) is
finitely generated abelian. Choose a decomposition
\(H_2(M;\Z)\cong\Z^g\oplus\tor(H_2(M;\Z))\).
The block-diagonal automorphisms acting identically on the
torsion summand form a finite-index subgroup isomorphic to
\(\GL_g(\Z)\). Let \(G_0\) be its inverse image in
\(\Mod(M)\).

The identification of the kernel and its action, together
with the rational bordism calculation in
\cite[Section~3]{JinSpinFive}, gives
\[
 V=K\otimes_\Z\R\cong
 \R^g\oplus\operatorname{Sym}^3(\R^g)
\]
as a \(\GL_g(\Z)\)-module. Here
\(\operatorname{Sym}^3\) denotes the third symmetric power.
Indeed, spin bordism with real coefficients in degree six of the
Eilenberg--MacLane space \(K(H_2(M;\Z),2)\) is the sum of
its ordinary homology in degrees two and six. These are
\(\R^g\) and \(\operatorname{Sym}^3(\R^g)\), respectively;
torsion in \(H_2(M;\Z)\) does not affect this calculation.
The identification is natural, so it also identifies the
conjugation action. 

The action extends to the unimodular Lie group
\[
 S=\{A\in\GL_g(\R):|\det A|=1\},
\]
which has property~\(\rt\) for \(g\ge3\) and contains
\(\GL_g(\Z)\) as a lattice. Its action on \(V\) has
absolute determinant one. The one-parameter subgroup
\[
 a_t=\operatorname{diag}(e^{-(g-1)t},e^t,\ldots,e^t)
 \in\SL_g(\R)
\]
contracts the first basis vector and its symmetric cube.
Conjugating \(a_t\) contracts any prescribed vector
\(v\in\R^g\) and its pure cube \(v^3\); pure cubes span
the symmetric power. Finally, the central element \(-I\)
acts as \(-I\) on both summands. Thus
\cref{thm:strategy}, with \(c=1\), gives property~\(\rt\)
for \(G_0\), since \(K\) is abelian. Finite-index permanence
gives the assertion for \(\Mod(M)\).

In particular, this applies to \(\#^g(S^2\times S^3)\) for
\(g\ge3\). In this case Jin proves that the extension
splits and that \(K\) is free abelian
\cite[Theorem~1.2]{JinSpinFive}. These mapping class groups
are linear: a semidirect product \(\Z^r\rtimes\GL_g(\Z)\)
embeds in integral matrices by combining its affine action
on \(\Z^r\) with the faithful standard representation of
\(\GL_g(\Z)\). The symmetric-cube summand distinguishes
this contraction argument from the standard symplectic
module occurring in \cref{sec:threefolds}.

\section{Non-residual finiteness of Torelli quotients}
\label{sec:nonresidual}

For \(g\ge3\) and \(c\ge2\), we prove that the Torelli
quotients \(\mathcal Q_{g,c}\) are not residually finite and
hence are not linear. We exhibit an infinite-order element
killed by every homomorphism to a finite group.

Deligne's theorem provides such elements in certain central
extensions of \(\Sp_{2g}(\Z)\) \cite{Deligne}. Hain's results
supply a candidate central element after removing 
torsion, and Hain--Matsumoto relate it to one of these
extensions \cite{Hain,HainMatsumoto}. We identify a subgroup
on which the comparison map is injective with finite-index
image in a Deligne extension. Arithmeticity allows us to find
a suitable symplectic subgroup in an algebraic image and
lift it back. Deligne's obstruction then transfers to our
quotient; a further power accounts for the finite torsion.

\subsection{Hain's central line}

Fix \(g\ge3\) and \(c\ge2\). The finitely generated nilpotent group
\(\mathcal N_{g,c}\) has a finite characteristic torsion subgroup
\(\tor(\mathcal N_{g,c})\), which is normal in \(\mathcal Q_{g,c}\).
We first consider the surjection
\[
 \mathcal Q_{g,c}/\tor(\mathcal N_{g,c})
       \longrightarrow\Sp_{2g}(\Z)
\]
whose kernel is \(\mathcal N_{g,c}/\tor(\mathcal N_{g,c})\).
Removing torsion will allow us to embed this nilpotent kernel
in its unipotent completion.

We recall the two completions needed to identify the central
element. Boldface letters denote algebraic groups over \(\Q\) and 
\(\Sp_{2g}\) denotes the symplectic algebraic group over \(\Q\).
For a field extension \(k/\Q\), \(U(k)\) denotes the group of \(k\)-rational points of
an algebraic or proalgebraic group \(U\). An algebraic group
is \emph{unipotent} if it embeds in an upper-triangular matrix
group with diagonal entries \(1\). Inverse limits of algebraic
groups are \emph{proalgebraic}, and inverse limits of unipotent
groups are \emph{prounipotent}. The additive algebraic group
is denoted by \(\Ga\), with \(\Ga(k)=(k,+)\).
There is a homomorphism
\(\mathcal T_g\to\mathscr P_g(\Q)\) into the rational points of the rational unipotent completion \(\mathscr P_g\) of
\(\mathcal T_g\). Every homomorphism
from \(\mathcal T_g\) to the rational points of a unipotent
group factors uniquely through a morphism from \(\mathscr P_g\).
\[
\begin{tikzcd}[column sep=large,row sep=large]
 \mathcal T_g \arrow[r] \arrow[dr,"f"']
   & \mathscr P_g(\Q) \arrow[d,dashed,"\exists!\,\widetilde f(\Q)"] \\
   & \mathbf U(\Q).
\end{tikzcd}
\]
Here \(\mathbf U\) is any unipotent algebraic group over \(\Q\),
\(f\) is any homomorphism, and the dashed arrow is induced by
the unique morphism \(\widetilde f:\mathscr P_g\to\mathbf U\)
over \(\Q\).

The completion \(\mathcal G_g\) of \(\Mod(\Sigma_g)\)
relative to its homological representation is an extension
of \(\Sp_{2g}\) by a prounipotent group \(\mathcal U_g\).
Concretely, \(\mathcal G_g\) is the inverse limit of algebraic
extensions of \(\Sp_{2g}\) by unipotent groups over \(\Q\),
equipped with lifts of the homological representation to
their rational points with Zariski-dense image. The connecting algebraic group homomorphisms preserve these lifts and the projections to \(\operatorname{Sp}_{2g}\). 
The group
\(\mathcal U_g=\ker(\mathcal G_g\to\Sp_{2g})\) is the
inverse limit of their unipotent kernels.
There is a homomorphism \(\operatorname{Mod}(\Sigma_g)\to\mathcal G_g(\mathbb Q)\) whose composition with the natural projection to \(\operatorname{Sp}_{2g}(\mathbb Q)\) is the homological representation. 
Every such lift to an algebraic extension
of \(\Sp_{2g}\) by a unipotent group factors uniquely through
a morphism from \(\mathcal G_g\), over \(\Sp_{2g}\)
\cite[Section~4]{HainMatsumoto}.

For any such algebraic extension \(\mathbf E\to\Sp_{2g}\)
and a lift to \(\mathbf E(\Q)\), the defining universal
property is expressed by the commutative diagram
\[
\begin{tikzcd}[column sep=large,row sep=large]
 \Mod(\Sigma_g) \arrow[r] \arrow[dr,"\text{given lift}"']
   & \mathcal G_g(\Q) \arrow[r]
       \arrow[d,dashed,"\exists!"]
   & \Sp_{2g}(\Q) \\
   & \mathbf E(\Q). \arrow[ur]
\end{tikzcd}
\]
where the dashed arrow is induced by the unique morphism
\(\mathcal G_g\to\mathbf E\) over \(\Sp_{2g}\).

Since the homological map kills \(\mathcal T_g\), these
universal properties give a natural map
\(\mathscr P_g\to\mathcal U_g\), as in the diagram
\[
\begin{tikzcd}[column sep=large,row sep=large]
 \mathcal T_g \arrow[rr,hook] \arrow[d] \arrow[dr]
   & {} & \Mod(\Sigma_g) \arrow[d] \\
 \mathscr P_g(\Q) \arrow[r,dashed,"\exists!"']
   & \mathcal U_g(\Q) \arrow[r,hook]
   & \mathcal G_g(\Q).
\end{tikzcd}
\]
Here the diagonal arrow is the restriction of the completion
map to \(\mathcal T_g\), with image in \(\mathcal U_g(\Q)\).
The dashed arrow is induced by the unique morphism
\(\mathscr P_g\to\mathcal U_g\) extending this restriction.

It was proved by Hain that this map is
surjective, with central kernel
\[
 \mathbf Z_g=\ker(\mathscr P_g\longrightarrow\mathcal U_g)
       \cong\Ga,\]
satisfying
\[
 \mathbf Z_g\subseteq(\mathscr P_g)_{[2]},\]
 and
 \[
 \mathbf Z_g\cap(\mathscr P_g)_{[3]}=\{1\}.
\]
Conjugation by \(\Mod(\Sigma_g)\) acts trivially on
\(\mathbf Z_g\)
\cite[Theorems~3.4 and~4.10]{Hain}. The inclusions above
show that \(\mathbf Z_g\) survives in every nilpotent stage of
class at least two. To pass to such a stage let
\[
 \mathbf N_c=\mathscr P_g/(\mathscr P_g)_{[c+1]},\]
 and
 \[
 \mathbf G_c=\mathcal G_g/(\mathcal U_g)_{[c+1]},
\]
where the lower central series is taken with closed subgroups.
These quotients are finite-dimensional because \(\mathcal T_g\)
is finitely generated. Denote the surviving central line by
\(\mathbf Z_c\subseteq\mathbf N_c\). The unipotent radical
of \(\mathbf G_c\) is then \(\mathbf N_c/\mathbf Z_c\).

It follows from Malcev theory that 
\(\mathcal N_{g,c}/\tor(\mathcal N_{g,c})\) embeds in
\(\mathbf N_c(\Q)\) as a lattice in \(\mathbf N_c(\R)\).
Since \(\mathbf Z_c\) is defined over \(\Q\), its intersection
with this lattice is a lattice in the real central line,
hence it is infinite cyclic. Fix a generator \(z\).
The image of the nilpotent group in the quotient is a lattice in
\((\mathbf N_c/\mathbf Z_c)(\R)\)
\cite[Chapter~II]{Raghunathan}.
The natural map from \(\Mod(\Sigma_g)\) to
\(\mathbf G_c(\Q)\) factors through a homomorphism
\[
 q_c:\mathcal Q_{g,c}/\tor(\mathcal N_{g,c})
             \longrightarrow\mathbf G_c(\Q).
\]
This map preserves the homological projection, and consequently  its kernel
lies in the image of \(\mathcal T_g\) under the composition 
\[
\operatorname{Mod}(\Sigma_g)
\longrightarrow\mathcal Q_{g,c}
\longrightarrow
\mathcal Q_{g,c}/\operatorname{tor}(\mathcal N_{g,c}).
\] There it kills precisely the
intersection with \(\mathbf Z_c(\Q)\), namely \(\langle z\rangle\).
Thus, letting \(\Lambda_c=\im q_c\), we obtain the central extension
\begin{equation}\label{eq:arithmetic-central-kernel}
 1\longrightarrow\langle z\rangle
 \longrightarrow\mathcal Q_{g,c}/\tor(\mathcal N_{g,c})
 \xrightarrow{q_c}\Lambda_c\longrightarrow1.
\end{equation}

\subsection{Arithmeticity}

We need arithmeticity of \(\Lambda_c\) to ensure that the intersection
of \(\Lambda_c\) with a rational symplectic subgroup considered later is sufficiently large.

A subgroup of the rational points of a linear algebraic group
is \emph{arithmetic} if, in a faithful rational representation,
it is commensurable with the subgroup preserving a full
integral lattice. Here commensurability means that the
intersection has finite index in both groups. This definition
is independent of the faithful representation.
We use the following standard arithmeticity criterion.

\begin{lemma}\label{lem:arithmetic-unipotent-extension}
Let \(\mathbf G\) be a connected linear algebraic group over
\(\Q\), with unipotent radical \(\mathbf U\) and quotient map
\(p:\mathbf G\to\Sp_{2g}\) identifying \(\mathbf G/\mathbf U\)
with \(\Sp_{2g}\). Suppose that
\(\Lambda\subseteq\mathbf G(\Q)\) is finitely generated,
\(p(\Lambda)=\Sp_{2g}(\Z)\), and
\(\Lambda\cap\mathbf U(\Q)\) is arithmetic in \(\mathbf U\).
Then \(\Lambda\) is arithmetic in \(\mathbf G\).
\end{lemma}
\begin{proof}
Choose a faithful rational representation \(\rho:\mathbf G\to\GL(V)\)
and a \(\mathbf G\)-invariant filtration
\(0=V_0\subseteq V_1\subseteq\cdots\subseteq V_r=V\) on which
\(\mathbf U\) acts trivially on successive quotients.
The quotient representations factor through \(\Sp_{2g}\),
so each preserves a lattice under \(\Sp_{2g}(\Z)\)
\cite[Chapter~VII, Proposition~5.1]{MilneArithmeticGroups}.
Their traces are therefore integral,
\[
\operatorname{tr}\!\left(\rho(\lambda)\big|_{V_i/V_{i-1}}\right)\in\mathbb Z,
\]
 and additivity of trace
gives integral traces for every element of \(\Lambda\).
Since \(\Lambda\) is finitely generated, the integral-trace
criterion \cite[Theorem~2.4]{BabaiBealsRockmore} implies
that it preserves a full integral lattice.

Adjoin the standard representation through \(p\), with
lattice \(\Z^{2g}\), and let \(\Gamma_0\) be the subgroup
of \(\mathbf G(\Q)\) preserving the resulting lattice.
Then \(\Lambda\subseteq\Gamma_0\) and
\(p(\Gamma_0)=p(\Lambda)=\Sp_{2g}(\Z)\).
Since the two groups have the same symplectic image, their
index is determined by their intersections with the kernel.
These intersections are arithmetic in \(\mathbf U\), hence
commensurable, and therefore
\[
 [\Gamma_0:\Lambda]
 =
 [\Gamma_0\cap\mathbf U(\Q):
                 \Lambda\cap\mathbf U(\Q)]<\infty,
\]
as required.
\end{proof}

The group \(\Lambda_c\) is finitely generated and projects
onto \(\Sp_{2g}(\Z)\). Its intersection with the unipotent
radical is the rational nilpotent lattice obtained above,
so the lemma proves its arithmeticity.

\subsection{A Deligne-type subgroup}

We now connect our quotient to Deligne's central extension.

\begin{theorem}[Deligne {\cite[p.~A203]{Deligne}}]
\label{thm:deligne-central-extension}
For \(g\ge2\), there exists a short exact sequence 
\[
 1\longrightarrow\langle u\rangle\cong\Z
 \longrightarrow\widetilde{\Sp_{2g}(\Z)}
 \longrightarrow\Sp_{2g}(\Z)\longrightarrow1.
\]
where 
\(\widetilde{\Sp_{2g}(\Z)}\)  is the inverse image of \(\Sp_{2g}(\Z)\) in the
universal covering group of \(\Sp_{2g}(\R)\).
Every finite-index subgroup of
\(\widetilde{\Sp_{2g}(\Z)}\) contains \(\langle u^2\rangle\).
\end{theorem}

For \(k\ne0\), the pushout along \(u\mapsto v^k\) is
\[
 \bigl(\widetilde{\Sp_{2g}(\Z)}\times\langle v\rangle\bigr)
       /\langle(u,v^{-k})\rangle,
 \qquad \langle v\rangle\cong\Z.
\]
This forces \(u=v^k\), while \(v\) remains of infinite order.
Consequently, every finite quotient of the above pushout kills
\(v^{2k}\). This is also true for finite quotients of any
finite-index subgroup. Indeed, the normal core of the kernel, the intersection of all of its
conjugates, is a
finite-index normal subgroup of the extension. It contains
\(v^{2k}\) and is contained in the original kernel.

The following construction of Hain and Matsumoto  connects this extension to
the mapping class group.

\begin{theorem}[Hain--Matsumoto
{\cite[Section~6.2 and Proposition~6.9]{HainMatsumoto}}]
\label{thm:hm-arithmetic-extension}
There is a central extension
\[
 1\longrightarrow\langle v\rangle\cong\Z
 \longrightarrow L^+\xrightarrow{p}
       A_g\rtimes\Sp_{2g}(\Z)\longrightarrow1
\]
and a homomorphism \(\Mod(\Sigma_g)\to L^+\) with
finite-index image \(L\). The image \(E\) of \(\mathcal T_g\)
is a torsion-free integral Heisenberg group that is the extension
\(1\to\Z\to E\to A_g\to1\), and the affine projection on
\(\mathcal T_g\) is twice the Johnson homomorphism.
The restriction
\[
 \Delta=p^{-1}\bigl(\{0\}\rtimes\Sp_{2g}(\Z)\bigr)
\]
is isomorphic to the the pushout of Deligne's extension along \(u\mapsto v^k\)
for a fixed nonzero integer \(k\).
\end{theorem}

The restriction gives the commutative diagram with exact rows
\[
\begin{tikzcd}[column sep=small,row sep=large]
 1 \arrow[r]
   & \langle v\rangle \arrow[r,hook] \arrow[d,equal]
   & \Delta \arrow[r,two heads] \arrow[d,hook]
   & \Sp_{2g}(\Z) \arrow[r]
       \arrow[d,hook,"{M\mapsto(0,M)}"]
   & 1 \\
 1 \arrow[r]
   & \langle v\rangle \arrow[r,hook]
   & L^+ \arrow[r,two heads,"p"']
   & A_g\rtimes\Sp_{2g}(\Z) \arrow[r]
   & 1.
\end{tikzcd}
\]
The right-hand square identifies \(\Delta\) with the pullback of \(p:L^+\to A_g\rtimes\operatorname{Sp}_{2g}(\mathbb Z)\) along the inclusion \(M\mapsto(0,M)\).

Here an integral Heisenberg group is a torsion-free group
of nilpotency class two with infinite cyclic center and
free abelian quotient by its center. In the affine projection,
the image of \(E\) is \(2A_g\), which is identified with
\(A_g\) in the above description of \(E\).

Since \(E\) is torsion-free and nilpotent of class two, the
homomorphism to \(L^+\) factors through
\[
 f_c:\mathcal Q_{g,c}/\tor(\mathcal N_{g,c})
                         \longrightarrow L^+.
\]
This map preserves Hain's line: on unipotent completions,
it identifies \(\mathbf Z_c\) with the center of
\(E^{\un}\), the unipotent completion of \(E\). The
Heisenberg quotient retains precisely this trivial symplectic
summand in degree two
\cite[Theorem~10.1]{Hain}
\cite[Section~6.2.1]{HainMatsumoto}. Consequently
\[
 f_c(z)=v^b\qquad\text{for an integer }b\ne0.
\]

Thus \(f_c\) is injective on \(\langle z\rangle\).
To extend this injectivity to a larger subgroup we use the
affine projection \(p f_c\). Compose this map with the quotient map from \(\Mod(\Sigma_g)\). 
The resulting map lifts the homological representation, so the universal
property of relative completion gives
\[
 \Phi_c:\mathbf G_c\longrightarrow
       (A_g\otimes\Q)\rtimes\Sp_{2g},\]
 where     
\[ \Phi_c q_c=p f_c.
\]
Here \(A_g\otimes\Q\) is regarded as an additive algebraic
group. The factorization through \(\mathbf G_c\) holds
because the target has abelian unipotent radical. The
restriction of \(\Phi_c\) to the unipotent radical is
surjective: the Torelli image \(2A_g\) is Zariski dense in
\(A_g\otimes\Q\).

A rational Levi subgroup of \(\mathbf G_c\) is a reductive
subgroup defined over \(\Q\) which projects isomorphically
onto \(\Sp_{2g}\). By Mostow's theorem, rational Levi
subgroups exist in characteristic zero
\cite[\S17.49]{MilneReductiveGroups}.

\begin{lemma}\label{lem:embedded-deligne}
There is a subgroup
\(D_c\subseteq\mathcal Q_{g,c}/\tor(\mathcal N_{g,c})\)
containing \(\langle z\rangle\) such that \(f_c|_{D_c}\)
is injective and its image has finite index in \(\Delta\).
\end{lemma}
\begin{proof}
Consider \(\Phi_c^{-1}(\{0\}\rtimes\Sp_{2g})\).
Its projection onto \(\Sp_{2g}\) is surjective. Indeed, working over \(\overline{\mathbb Q}\), choose a lift \(x\) of a symplectic matrix \(M\) and write \(\Phi_c(x)=(a,M)\). Surjectivity on unipotent radicals gives an element \(u\) in the unipotent radical of \(\mathbf G_c\) with \(\Phi_c(u)=(-a,I)\). Then \(\Phi_c(ux)=(0,M)\), while \(ux\) still projects to \(M\), proving the required surjectivity.

The kernel of this projection is unipotent.
Choose a rational Levi subgroup \(\mathbf S_c\) of this
extension. 
The extension is the product of its unipotent kernel and
\(\mathbf S_c\), so the projection remains surjective on \(\mathbf S_c\).
\(\mathbf S_c\) intersects the unipotent kernel trivially, and thus the projection \(\mathbf S_c\to\operatorname{Sp}_{2g}\) is injective. Together with surjectivity, this shows that the projection is an isomorphism of algebraic groups over \(\mathbb Q\).
By construction we have
\[
 \Phi_c(\mathbf S_c)=\{0\}\rtimes\Sp_{2g}.
\]
Now take 
\[
 \Gamma_c=\Lambda_c\cap\mathbf S_c(\Q),\]
 and 
 \[ D_c=q_c^{-1}(\Gamma_c).
\]
Identify \(\mathbf S_c\) with \(\Sp_{2g}\) using the
homological projection. Since arithmeticity is preserved
under intersection with rational algebraic subgroups,
\(\Gamma_c\) is commensurable with \(\Sp_{2g}(\Z)\).
Since it is contained in that group, it in fact has finite index.
Moreover, \(\Phi_c q_c=p f_c\) and the choice of
\(\mathbf S_c\) give \(f_c(D_c)\subseteq\Delta\).
We thus obtain the commutative diagram with exact rows
\[
\begin{tikzcd}[column sep=small,row sep=large]
 1 \arrow[r]
   & \langle z\rangle \arrow[r,hook]
       \arrow[d,hook,"z\mapsto v^b"']
   & D_c \arrow[r,two heads] \arrow[d,"f_c"]
   & \Gamma_c \arrow[r] \arrow[d,hook]
   & 1 \\
 1 \arrow[r]
   & \langle v\rangle \arrow[r,hook]
   & \Delta \arrow[r,two heads]
   & \Sp_{2g}(\Z) \arrow[r]
   & 1.
\end{tikzcd}
\]
Both outer vertical maps are injective, so exactness implies
that \(f_c|_{D_c}\) is injective. Their images have finite
index. More precisely,
\[
 [\Delta:f_c(D_c)]
   =[\Sp_{2g}(\Z):\Gamma_c]\,|b|<\infty.
\]
\end{proof}

\subsection{Proof of Theorem D}

\begin{proof}[Proof of \cref{thm:nonresidual}]
Let  \(M=2|k|\), with \(k\) as in
\cref{thm:hm-arithmetic-extension}. Every finite quotient
of \(\Delta\) kills \(v^M\), hence also \(v^{bM}\).
For any homomorphism \(\psi:f_c(D_c)\to F\) to a finite group, the normal core of \(\ker\psi\) in \(\Delta\) has finite index in \(\Delta\) and is contained in \(\ker\psi\). The corresponding finite quotient of \(\Delta\) kills \(v^{bM}\), so \(v^{bM}\) belongs to this core and therefore \(\psi(v^{bM})=1\).
 Since \(f_c\) is
injective on \(D_c\) and \(f_c(z^M)=v^{bM}\), every finite
quotient of \(D_c\) kills \(z^M\). Restricting to \(D_c\)
now shows that every homomorphism from
\(\mathcal Q_{g,c}/\tor(\mathcal N_{g,c})\) to a finite
group kills this same fixed power.

It remains to restore torsion. Choose a lift
\(y\in\mathcal Q_{g,c}\) of \(z^M\), and put
\(e=|\tor(\mathcal N_{g,c})|\). For any surjection
\(h:\mathcal Q_{g,c}\to F\) onto a finite group, the induced
map to \(F/h(\tor(\mathcal N_{g,c}))\) factors through the
quotient by \(\tor(\mathcal N_{g,c})\) and therefore kills \(y\). Hence
\(h(y)\in h(\tor(\mathcal N_{g,c}))\), and \(h(y^e)=1\).
Since \(y^e\) maps to the infinite-order element \(z^{Me}\),
it too has infinite order. This proves \cref{thm:nonresidual}.
\end{proof}

We remark, that the above  argument requires \(c\ge2\), as Hain's central line
disappears in degree one.

\begin{corollary}\label{cor:nonresidual-nonlinearity}
Let \(g\ge3\) and \(c\ge2\). Every finite-dimensional linear
representation of \(\mathcal Q_{g,c}\) over any field has
infinite kernel. In particular, \(\mathcal Q_{g,c}\) is not
linear over any field. The same conclusions hold for every
finite-index subgroup.
\end{corollary}
\begin{proof}
The normal-core argument shows that the infinite cyclic
subgroup constructed above lies in every finite-index
subgroup and is killed by all its finite quotients.
Every linear image of a finite-index subgroup of
\(\mathcal Q_{g,c}\) is finitely generated and hence
residually finite \cite{Nica}. The representation therefore
kills this infinite cyclic subgroup.
\end{proof}

\section{Final remarks}

\subsection{Expansion and dependence on nilpotency class}

Suppose that \(G\) is finitely generated and that
\(G/K_{[c+1]}\) has property~\(\rt\) for every \(c\ge1\).
Fix a finite symmetric generating set \(S\) of \(G\) and fix \(c\).
Every finite quotient of \(G\) in which the image of \(K\)
is nilpotent of class at most \(c\) factors through the same
Kazhdan group \(G/K_{[c+1]}\).
Consequently the lazy random walks on these quotients, defined
by the images of \(S\) with multiplicities retained, have a
common positive spectral gap away from constant functions.
Thus their Cayley graphs form an expander family whenever
their orders tend to infinity, see \cite[Chapter~6]{BHV}.

The dependence of this gap on \(c\) is a separate question.
For free nilpotent groups, product-replacement applications
were established by \cite{LubotzkyPak}, and explicit
depth-dependent Kazhdan bounds for tame automorphism groups
appear in \cite[Theorem~3.14]{HadadExtensions}.
Obtaining comparable bounds for the larger quotients considered
here, or bounds independent of \(c\), would strengthen the
qualitative permanence argument. Property~\(\rt\) at each
stage does not itself give a uniform bound. Conversely,
failure of property~\(\rt\) for the original group does not
force the quotient gaps to tend to zero: uniform expansion
of a family of quotients is weaker than property~\(\rt\)
of the original group.

\subsection{Failure of uniform property \(\rt\)}

For \(g\ge3\) and \(c\ge1\), the groups \(\mathcal Q_{g,c}\)
do not have uniform property~\(\rt\). Write \(\kappa_{\mathrm u}(G)\)
for the infimum of the Kazhdan constants of \(G\) over all finite
generating sets. Since \(\mathcal Q_{g,c}\) surjects onto
\(\Sp_{2g}(\Z)\), the quotient estimate and the recent theorem of
Lubotzky and Yao \cite[Lemma~2.1 and Theorem~1.1]{LubotzkyYao} give
\[
 \kappa_{\mathrm u}(\mathcal Q_{g,c})
 \le \kappa_{\mathrm u}(\Sp_{2g}(\Z))=0.
\]

\subsection*{Acknowledgements}
The author was supported by the National Science Center Grant Maestro-13
UMO-2021/42/A/ST1/00306. 

\subsection*{Statement on AI use}
The author used GPT-Sol 5.6 and GPT-Astra 6 during the development of this work to explore proof strategies and examples, examine mathematical arguments, and assist with drafting and revising the manuscript. 
The author assumes responsibility for all content, including the verification of the proofs and the final presentation.

\bigskip
\noindent

\noindent\small \textsc{Piotr W. Nowak}\\ \textsc{Institute of Mathematics of the Polish Academy of Sciences, Warsaw, Poland.}\\
\textit{Email:} \texttt{pnowak@impan.pl}\\[0.5em]

\end{document}